\documentclass[12pt,english]{amsart}
\usepackage[a4paper]{geometry}
\usepackage{amssymb,amsmath}
\usepackage{amsthm}
\usepackage{enumerate}
\usepackage[utf8]{inputenc}
\usepackage[english]{babel}
\usepackage{color}

\numberwithin{equation}{section}

\theoremstyle{plain}
\newtheorem{theorem}{Theorem}[section]

\newtheorem{proposition}[theorem]{Proposition}
\newtheorem{lemma}[theorem]{Lemma}
\newtheorem{corollary}[theorem]{Corollary}

\theoremstyle{definition}

\newtheorem{remark}[theorem]{Remark}
\newtheorem{example}[theorem]{Example}
\newtheorem{question}[theorem]{Question}

\newcommand{\C}{\mathbb{C}}
\newcommand{\D}{\mathbb D}

\newcommand{\N}{\mathbb{N}}
\newcommand{\R}{\mathbb{R}}

\newcommand{\UU}{\mathcal U}

\renewcommand{\tilde}{\widetilde}

\usepackage[bookmarksopen, naturalnames]{hyperref}

\begin{document}

\title{Chaos and frequent hypercyclicity for weighted shifts}
\author{St\'ephane Charpentier, Karl Grosse-Erdmann, Quentin Menet}
\address{St\'ephane Charpentier, 
Institut de Math\'ematiques de Marseille, UMR 7373, Aix-Marseille Universit\'e, 39 rue F. Joliot Curie, 13453 Marseille Cedex 13, FRANCE}
\email{stephane.charpentier.1@univ-amu.fr}
\address{Karl Grosse-Erdmann, 
Département de Math\'ematique, Universit\'e de Mons, 20 Place du Parc, 7000 Mons, BELGIUM}
\email{kg.grosse-erdmann@umons.ac.be}
\address{Quentin Menet, Département de Mathématique, Université de Mons, 20 Place du Parc, 7000 Mons, BELGIUM}
\email{quentin.menet@umons.ac.be}
\thanks{St\'ephane Charpentier and Quentin Menet were supported by the grant ANR-17-CE40-0021 of the French National Research Agency ANR (project Front). Karl Grosse-Erdmann was supported by FNRS grant PDR T.0164.16. Quentin Menet is a Research Associate of the Fonds de la Recherche Scientifique - FNRS}
\keywords{Frequently hypercyclic operator, chaotic operator, weighted shift operator, K\"othe sequence space, power series space}
\subjclass[2010]{47A16, 47B37, 46A45}

\begin{abstract}
Bayart and Ruzsa [\textit{Ergodic Theory Dynam. Systems} 35 (2015)] have recently shown that every frequently hypercyclic weight\-ed shift on $\ell^p$ is chaotic. This contrasts with an earlier result of Bayart and Grivaux [\emph{Proc. London Math. Soc.}~(3) 94 (2007)] who constructed a non-chaotic frequently hypercyclic weighted shift on $c_0$. We first generalize the Bayart-Ruzsa theorem to all Banach sequence spaces in which the unit sequences are a boundedly complete unconditional basis. We then study the relationship between frequent hypercyclicity and chaos for weighted shifts on Fr\'echet sequence spaces, in particular on K\"othe sequence spaces, and then on the special class of power series spaces. We obtain, rather curiously, that every frequently hypercyclic weighted shift on $H(\mathbb{D})$ is chaotic, while $H(\mathbb{C})$ admits a non-chaotic frequently hypercyclic weighted shift.
\end{abstract}

\maketitle

\section{Introduction}\label{sec1}

Chaos and frequent hypercyclicity are two of the most important notions in linear dynamics. An operator $T$ on a separable Fr\'echet space $X$ is called chaotic if it admits a dense orbit (i.e., it is hypercyclic) and if it has a dense set of periodic points. If one demands from an orbit that it combines both aspects of chaos, i.e., that it is dense and that it returns often, one arrives at the notion of frequent hypercyclicity: A vector $x\in X$ is called frequently hypercyclic for $T$ if, for any non-empty open set $U$ in $X$, 
\[
\underline{\text{dens}}\{n\geq 0 : T^n x\in U\} >0;
\]
recall that, for a subset $A\subset \mathbb{N}_0$, $\underline{\text{dens}}(A) = \liminf_{N\to\infty}\frac{1}{N+1}\text{card}\{n\leq N : n\in A\}$ denotes its lower (asymptotic) density. The operator $T$ is called frequently hypercyclic if it admits a frequently hypercyclic vector. This notion was introduced by Bayart and Grivaux \cite{BaGr04}, \cite{BaGr06}. For more information on linear dynamics we refer to \cite{BaMa09} and \cite{GrPe11}.

The precise relationship between chaos and frequent hypercyclicity has been intriguing researchers for the last decade. Despite appearances, these notions have turned out to be independent. First, Bayart and Grivaux \cite{BaGr07} showed that a frequently hypercyclic operator need not be chaotic. Indeed, they constructed a frequently hypercyclic weighted backward shift on $c_0$ that does not even have a single non-trivial periodic point; see also \cite{BoGr18}. Recently, the third author \cite{Men17} constructed operators on any of the spaces $\ell^p$, $1\leq p<\infty$, and $c_0$ that are chaotic but not frequently hypercyclic. 

One of the best understood classes of operators in linear dynamics is that of weighted (backward) shifts on sequence spaces. Given a sequence $w=(w_n)_n$ of non-zero scalars, the corresponding weighted shift operator is formally given by
\[
B_w (x_n)_{n\geq 0} = (w_{n+1}x_{n+1})_{n\geq 0}.
\]
For these operators, under suitable assumptions, chaos does imply frequent hypercyclicity.

Indeed, let $X$ be a Fr\'echet sequence space, that is, a Fr\'echet space of (real or complex) sequences $x=(x_n)_{n\geq 0}$ such that each coordinate functional $x\mapsto x_n$, $n\geq 0$, is continuous. By the closed graph theorem, a weighted  shift $B_w$ defines a (continuous) operator on $X$ as soon as it maps the space $X$ into itself. Let $e_n$, $n\geq 0$, denote the unit sequences. Under the assumption that $(e_n)_n$ is an unconditional basis in the Fr\'echet sequence space $X$ then a weighted  shift $B_{w}$ on $X$ is known to be chaotic if and only if the series
\begin{equation}\label{eq0}
\sum _{n\geq 0} \frac{1}{w_1\cdots w_n}e_n
\end{equation}
converges in $X$; see \cite{Gro00}. It then follows from the Frequent Hypercyclicity Criterion that, in this case, $B_w$ is also frequently hypercyclic; see \cite[Corollary 9.14]{GrPe11}.

We know from the result of Bayart and Grivaux cited above that not every frequently hypercyclic weighted shift on $c_0$ is chaotic. In light of this, the following recent result of Bayart and Ruzsa \cite{BaRu15} came as a surprise.

\begin{theorem}[Bayart, Ruzsa]\label{thm-bayruz}
Let $1\leq p<\infty$ and let $w=(w_n)_{n}$ be a bounded sequence of non-zero scalars. Then the weighted shift $B_{w}$ is frequently hypercyclic on $\ell^p$ if and only if it is chaotic.
\end{theorem}

So what do the spaces $\ell^p$ possess that $c_0$ does not? This question was the starting point of our present investigation. In Section \ref{sec2} we will show that the result of Bayart and Ruzsa extends to all Banach sequence spaces for which $(e_n)_n$ is a boundedly complete unconditional basis. Going beyond the class of Banach spaces we then turned our attention to K\"othe sequence spaces, which are important generalizations of the spaces $\ell^p$ and $c_0$. The K\"othe sequence space of order $p$ is defined as
\[
\lambda^p(A) = \Big\{ x=(x_n)_{n\geq 0} : \text{for all $m\geq 1$, } \sum_{n\geq 0} |x_n|^pa_{m,n}<\infty\Big\},
\] 
while the K\"othe sequence space of order $0$ is given by
\[
c_0(A) = \Big\{ x=(x_n)_{n\geq 0} : \text{for all $m\geq 1$, } \lim_{n\to\infty}|x_n|a_{m,n}=0\Big\},
\] 
where $A=(a_{m,n})_{m\geq 1,n\geq 0}$ is a matrix of strictly positive numbers such that, for all $m\geq 1$, $n\geq 0$, $
a_{m,n}\leq a_{m+1,n}$. In particular, if $a_{m,n}=1$ for every $m, n$ then $\lambda^p(A)=\ell^p$ and $c_0(A)=c_0$. Thus one could naively imagine that on K\"othe sequence spaces of order $p$ each frequently hypercyclic weighted shift is chaotic, while on K\"othe sequence spaces of order $0$ there exists a frequently hypercyclic weighted shift that is not chaotic. However, it is not seldom that $\lambda^p(A)= c_0(A)$ for any $1\le p<\infty$. This is, for instance, the case for the space $H(\C)$ of entire functions and for the space $H(\D)$ of holomorphic functions on the unit disk. The link between chaotic and frequently hypercyclic weighted shifts on these spaces is therefore particularly intriguing.

In order to be able to deal with these spaces, the main part of the paper is devoted to Fr\'echet sequence spaces, in varying degrees of generality. Two interesting phenomena arise. First, as we will see in Section \ref{sec3}, there are simple Fr\'echet sequence spaces that do not support any weighted shift, and there are natural K\"othe sequence spaces that do not support any hypercyclic weighted shift, so that our study may become vacuous. Secondly, bounded completeness is no longer sufficient for ensuring that frequent hypercyclicity implies chaos. Still, in Section \ref{sec2} we obtain a version of the Bayart-Ruzsa theorem for a class of Fr\'echet sequence spaces. In Section \ref{sec3} we spell out this result in the case of K\"othe sequence spaces, and we state sufficient conditions for the existence of a frequently hypercyclic weighted shift that is not chaotic. This will allow us to obtain the puzzling result that every frequently hypercyclic weighted shift on $H(\D)$ is chaotic, while there exists a frequently hypercyclic weighted shift on $H(\C)$ that is not chaotic. 

Finally, in Section~\ref{sec-power}, we generalize these two examples by studying the particularly important class of power series spaces (of finite or infinite type). We will obtain a full characterization of when every frequently hypercyclic weighted shift is chaotic on a class of power series spaces that satisfy some regularity assumption. In particular, it will appear that, on the one hand, there are Köthe sequence spaces $c_0(A)$, different from any $\lambda^p(A)$, $1\leq p <\infty$, on which any frequently hypercyclic weighted shift is chaotic; on the other hand, there are Köthe sequence spaces $\lambda^p(A)$, different from $c_0(A)$, on which there exists a non-chaotic frequently hypercyclic weighted shift.

\section{When frequent hypercyclicity implies chaos for weighted shifts}\label{sec2}

\subsection{The Bayart-Ruzsa theorem on Banach sequence spaces}\label{subsec-Banach}

The purpose of this subsection is to extend Theorem \ref{thm-bayruz} to a large class of Banach sequence spaces. Before stating this result we remark that Bayart and Ruzsa obtained a stronger version of their result: even $\mathcal{U}$-frequent hypercyclicity implies chaos (see \cite[Theorem 4]{BaRu15}). An operator $T$ on a Fr\'echet space $X$ is called $\mathcal{U}$-frequently hypercyclic if there is vector $x\in X$, also called $\mathcal{U}$-frequently hypercyclic, such that, for any non-empty open set $U$ in $X$, 
\[
\overline{\text{dens}}\{n\geq 0 : T^n x\in U\} >0,
\]
where, for a subset $A\subset \mathbb{N}_0$, $\overline{\text{dens}}(A) = \limsup_{N\to\infty}\frac{1}{N+1}\text{card}\{n\leq N : n\in A\}$ denotes its upper (asymptotic) density. This notion is obviously weaker than frequent hypercyclicity.

In our result, as is usual in the context of weighted  shifts, we will demand that $(e_n)_{n}$ is an unconditional basis, which is the case both for $c_0$ and the spaces $\ell^p$. What distinguishes these spaces, however, is that $(e_n)_{n}$ is boundedly complete in $\ell^p$ but not in $c_0$. Recall that a basis $(f_n)_n$ in a Fr\'echet space $X$ is called \textit{boundedly complete} if, for any sequence of scalars $x=(x_n)_n$, whenever the sequence
\[
\Big(\sum _{n=0}^N x_n f_n\Big)_{N\geq 0}
\]
is bounded in $X$ then it converges in $X$; see \cite{AlKa06}. 

\begin{theorem}\label{thm-Banach}
Let $X$ be a Banach sequence space for which $(e_n)_n$ is a boundedly complete unconditional basis. Let $B_w$ be a weighted shift on $X$. Then the following assertions are equivalent:
\begin{enumerate}
\item[\rm (i)] $B_{w}$ is $\UU$-frequently hypercyclic on $X$;
\item[\rm (ii)] $B_{w}$ is frequently hypercyclic on $X$;
\item[\rm (iii)] $B_{w}$ is chaotic on $X$;
\item[\rm (iv)] the series $\sum _{n\geq 0}\frac{1}{w_1\cdots w_n}e_n$ is convergent in $X$.
\end{enumerate}
\end{theorem}

As we have mentioned in the introduction, the implications {(iv)}$\Longrightarrow$(iii)$\Longrightarrow$(ii) hold in all Banach sequence spaces in which $(e_n)_n$ is an unconditional basis, while the implication (ii)$\Longrightarrow$(i) is trivial. It remains therefore to show that (i) implies (iv) under bounded completeness. 

The following lemma will be crucial in the proof. A subset $A\subset\N_0$ is called syndetic if it is infinite and of bounded gaps, that is, $\sup_{n\in A} \inf_{m\in A, m> n} (m-n)<\infty$.

\begin{lemma}\label{lemma1}
Let $X$ be a Banach space with an unconditional basis $(f_n)_{n\geq 0}$. Let $(\alpha _n)_{n\geq 0}$ be a sequence of scalars and $A\subset\N_0$ a subset of positive upper density such that, for some $R>0$, the family
\[
\Big(\frac{1}{N+1}\sum_{\substack{n\in A\\n\leq N}}\sum _{\substack{m\in A\\n\leq m\leq n+M}}\alpha _{m-n}f_{m-n}\Big)_{M,N\geq R}
\]
is bounded in $X$. Then there is a syndetic set $F\subset\N_0$ such that the sequence
\[
\Big(\sum_{n\in F,\,n\leq N} \alpha_n f_n\Big)_{N\geq 0}
\]
is bounded in $X$. 
\end{lemma}

The proof is based on Theorem \ref{thm-erdos} below, which is an improvement of a result due to Erd\H{o}s and S\'ark\"ozy that was implicitly obtained by Bayart and Ruzsa in order to prove Theorem \ref{thm-bayruz}. Note that these authors only state a weaker version \cite[Theorem 8]{BaRu15}.

\begin{theorem}[Bayart, Ruzsa]\label{thm-erdos}
Let $A\subset \N_0$ be a set of upper density $\delta:=\overline{\text{\emph{dens}}}(A)>0$, and let $0<\varepsilon <\delta^2$. Then there exists a strictly increasing sequence $(N_j)_j$ of positive integers such that the set
\[
\Big\{k\in \N_0:\quad \lim_{j\to\infty}\frac{\text{\emph{card}}\{n\leq N_j:n\in A\cap(A-k)\}}{N_j+1}>\delta^2-\varepsilon\Big\}
\]
is syndetic.
\end{theorem}

\begin{proof}[Proof of Lemma \ref{lemma1}]
Let $(\alpha _n)_n$ and $A\subset \N_0$ be as in the statement of the lemma, with $\delta:=\overline{\text{dens}}(A)>0$, and fix $0<\varepsilon <\delta ^2$. By Theorem \ref{thm-erdos} there exists a strictly increasing sequence $(N_j)_j$ of positive integers such that the set
\[
F:=\Big\{k\in \N_0:\quad \lim_{j\to\infty}\frac{\text{{card}}\{n\leq N_j:n\in A\cap(A-k)\}}{N_j+1}>\delta^2-\varepsilon\Big\}
\]
is syndetic. For $N, M\geq 0$ we set
\[
y_{N,M}=\frac{1}{N+1}\sum_{\substack{n\in A\\n\leq N}}\sum _{\substack{m\in A\\n\leq m\leq n+M}}\alpha _{m-n}f_{m-n}.
\]
Reordering the (finite) sum, we get
\begin{equation}\label{eq1}
y_{N,M} = \frac{1}{N+1} \sum_{k=0}^M\sum _{\substack{n\in A,\,n\in A-k\\n\leq N}}\alpha_{k}f_{k} =  \sum_{k=0}^M\frac{\text{card}\{n\leq N:n\in A\cap(A-k)\}}{N+1}\alpha _{k}f_{k}.
\end{equation}

Since $(f_n)_n$ is an unconditional basis there exists a constant $C>0$ such that, whenever $x=\sum_{n\geq 0} x_n f_n \in X$ and $b=(b_n)_n$ is a bounded sequence of scalars, then $\sum _{n\geq 0} b_nx_nf_n\in X$ and
\begin{equation}\label{eq2}
\Big\Vert \sum _{n\geq 0} b_nx_nf_n\Big\Vert\leq C \|b\|_\infty\Big\Vert \sum _{n\geq 0} x_nf_n\Big\Vert,
\end{equation}
see \cite[Proposition 3.1.3]{AlKa06}. Taking for $b$ a suitable $0$-$1$-sequence we obtain from \eqref{eq1} and \eqref{eq2} that
\begin{equation}\label{eq3}
\|y_{N,M}\| \geq \frac{1}{C}\Big\| \sum_{k\in F,\,k\leq M}\frac{\text{card}\{n\leq N:n\in A\cap(A-k)\}}{N+1}\alpha _{k}f_{k}\Big\|.
\end{equation}

By definition of $F$ and another application of unconditionality we have that, for any $M\geq 0$ and for all sufficiently large $j$,
\begin{align}\label{eq4}
\Big\Vert \sum_{k\in F,\,k\leq M}\frac{\text{card}\{n\leq N_j:n\in A\cap(A-k)\}}{N_j+1}\alpha _{k}f_{k}\Big\Vert \geq \frac{\delta^2-\varepsilon}{C}\Big\Vert \sum_{k\in F,\,k\leq M}\alpha _{k}f_{k}\Big\Vert.
\end{align}

It then follows from \eqref{eq3} and \eqref{eq4} that for any $M\geq 0$ and all sufficiently large $j$, 
\[
\Big\Vert \sum_{k\in F,\,k\leq M}\alpha _{k}f_{k}\Big\Vert\leq \frac{C^2}{\delta^2-\varepsilon} \Vert y_{N_j,M}\Vert.
\]
The hypothesis now implies that
\[
\Big(\sum_{k\in F,\,k\leq M}\alpha _{k}f_{k}\Big)_{M\geq R}
\]
is bounded in $X$.
\end{proof}

\begin{remark}\label{rem1}
(a) Obviously, the hypothesis of the lemma holds in particular if the family
\[
\Big(\sum _{\substack{m\in A\\n\leq m\leq n+M}}\alpha _{m-n}f_{m-n}\Big)_{n\in A,M\geq 0}
\]
is bounded in $X$. This is the condition that will be satisfied in our application of the lemma.

(b) Using essentially the same proof one can show that the lemma also holds for Fr\'echet spaces. However, in that setting the hypothesis is too strong for our intended application, see the proof of Theorem \ref{thm-Frechet}. 
\end{remark}

We will apply Lemma \ref{lemma1} to the unit sequence $(e_n)_n$ in a Banach sequence space $X$ and to a sequence $(\alpha_n)_n$ that arises in a natural way from the weights of a weighted shift $B_w$. In fact, as is usual in this context, we will define a sequence $v=(v_n)_{n\geq 0}$ by
\begin{equation}\label{eq_v}
v_n=\frac{1}{w_1\cdots w_n},\ n\geq 0;
\end{equation}
note that $v_0=1$. Conversely, given a sequence $v=(v_n)_{n\geq 0}$ of non-zero scalars with $v_0=1$, we recover the sequence $w$ by setting
\begin{equation}\label{eq_w}
w_n = \frac{v_{n-1}}{v_n},\quad n\geq 1.
\end{equation}

\begin{lemma}\label{lemma2}
Let $X$ be a Banach sequence space for which $(e_n)_n$ is an unconditional basis. Let $B_w$ be a weighted shift on $X$, and let $v$ be the sequence associated to $w$ by \eqref{eq_v}.
If $A\subset\N_0$ is a subset of positive upper density such that, for some $R>0$, the family
\[
\Big(\frac{1}{N+1}\sum_{\substack{n\in A\\n\leq N}}\sum _{\substack{m\in A\\n\leq m\leq n+M}}v_{m-n}e_{m-n}\Big)_{M,N\geq R}
\]
is bounded in $X$, then the sequence
\[
\Big(\sum_{n\leq N} v_n e_n\Big)_{N\geq 0}
\]
is bounded in $X$. 
\end{lemma}

\begin{proof} We know from Lemma \ref{lemma1} that
\[
\Big(\sum_{n\in F,\,n\leq N} v_n e_n\Big)_{N\geq 0}
\]
is bounded in $X$ for some syndetic set $F\subset\N_0$. By the definition of $B_w$ we have that
\[
B_w\Big(\sum_{n\in F,\,0\leq n\leq N} v_n e_n\Big) = \sum_{n\in F,\,0\leq n\leq N}v_n B_w e_n = \sum_{n\in F,\,1\leq n\leq N}v_n w_n e_{n-1} =\sum_{n\in F,\,1\leq n\leq N} v_{n-1} e_{n-1}.
\]
Now, since $F$ is syndetic there is some $K\geq 0$ such that
\[
\N_0=\bigcup_{k\leq K}  (F-k),
\]
which implies that
\[ 
\sum_{k=0}^K B_w^k \Big(\sum_{n\in F,\,0\leq n\leq N} v_n e_n\Big) = \sum_{0\leq n\leq F_N} m_{n,N}v_n e_n
\]
with certain integers $ m_{n,N}\geq 1$, where $F_N=\max\{n\in F: n\leq N\}$, with $\max\varnothing =-1$. 

The result now follows by the continuity of $B_w$, the unconditionality of the basis and the fact that $F_N\to \infty$ as $N\to\infty$.
\end{proof}

\begin{remark}\label{rem-gaps} Instead of filling the gaps of $F$ via a backward shift, as done in the proof of the lemma, one may also fill them by a forward shift. Let $w=(w_n)_{n\geq 0}$ be a sequence of non-zero scalars. The corresponding forward shift $F_w$ is defined by $F_{w}x=\sum_{n=1}^\infty w_{n-1}x_{n-1}e_{n}$. Then the preceding lemma also holds if $F_w$ is an operator on $X$ and $v$ satisfies $w_n=\frac{v_{n+1}}{v_n}$, $n\geq 0$.
\end{remark}

We are now ready to prove Theorem \ref{thm-Banach}.

\begin{proof}[Proof of Theorem \ref{thm-Banach}]
As mentioned before, it suffices to show that (i) implies (iv).

Let $x=(x_n)_n \in X$ be a $\UU$-frequently hypercyclic vector for $B_{w}$. Let $C$ be a constant due to unconditionality appearing in \eqref{eq2}. Then the set
\[
A:=\Big\{n\in \N_0 :\Vert B_{w}^nx-e_0\Vert \leq \frac{\Vert e_0\Vert}{2C}\Big\}
\]
has positive upper density. Since
\begin{equation}\label{def-bw}
B_{w}^nx=\sum _{k\geq 0}w_{k+1}\cdots w_{k+n}x_{k+n}e_k= \sum _{k\geq 0}\frac{v_k}{v_{k+n}}x_{k+n}e_k,
\end{equation}
where  $v$ is the sequence associated to $w$ by \eqref{eq_v}, we have that for every $n\in A$
\[
\Big\Vert \frac{1}{v_n}x_{n}e_0-e_0 \Big\Vert\leq C\Vert B_{w}^nx-e_0\Vert \leq \frac{\Vert e_0\Vert}{2}.
\]
Hence
\begin{equation}\label{eq2-thm-Banach}
\Big|\frac{x_n}{v_n}\Big|\geq 1/2
\end{equation}
for every $n\in A$.

By \eqref{def-bw} and \eqref{eq2-thm-Banach}, we note that for every $n\in A$ and every integer $M\geq 0$,
\begin{align*}
\Big\Vert\sum_{\substack{m\in A\\n\leq m\leq n+M}}v_{m-n}e_{m-n}\Big\Vert &\leq 2 C\Big\Vert\sum_{\substack{m\in A\\n\leq m\leq n+M}}v_{m-n}\frac{x_m}{v_m}e_{m-n}\Big\Vert= 2 C\Big\Vert\sum_{\substack{0\leq k\leq M\\k+n\in A}}v_{k}\frac{x_{k+n}}{v_{k+n}}e_{k}\Big\Vert\\
&\leq 2C^2 \Vert B_w^nx\Vert \leq  2C^2 (\Vert B_w^nx-e_0\Vert+\Vert e_0\Vert)\\
&\leq C(1+2C)\Vert e_0\Vert.
\end{align*}
Applying now Lemma \ref{lemma2}, taking account of Remark \ref{rem1}(a),
we obtain that 
\[
\Big(\sum_{n\leq N}\frac{1}{w_{1}\cdots w_{n}}e_{n}\Big)_{N\geq 0}=\Big(\sum_{n\leq N}v_ne_{n}\Big)_{N\geq 0}
\]
is bounded in $X$. Since the basis $(e_n)_n$ is boundedly complete we can finally deduce (iv).
\end{proof}

\begin{remark}\label{rem-genufhc}
Since the proof uses bounded completeness of the basis only in the last step, we have also proved the following. Let $X$ be a Banach sequence space for which $(e_n)_n$ is an unconditional basis. If a weighted shift $B_w$ on $X$ is $\UU$-frequently hypercyclic then $(\sum_{n\leq N}\frac{1}{w_{1}\cdots w_{n}}e_{n})_{N\geq 0}$ is bounded in $X$.
\end{remark}

Can one go beyond bounded completeness? There is an indication that this might be difficult: it is known that if $(e_n)_n$ is an unconditional basis in a Banach sequence space $X$, then it fails to be boundedly complete if and only if $X$ contains a copy of $c_0$ (see \cite[Theorem 3.3.2]{AlKa06} for the precise statement), and we know from Bayart and Grivaux \cite{BaGr07} that Theorem \ref{thm-Banach} fails on $c_0$. But it is not clear how to exploit these facts.

\begin{question}\label{q-bddcompl}
Is there a Banach sequence space $X$ for which $(e_n)_n$ is a non-boundedly complete unconditional basis so that every frequently hypercyclic weighted shift on $X$ is chaotic?
\end{question}

\subsection{The Bayart-Ruzsa theorem on Fr\'echet sequence spaces}\label{subsec-Frechet}

Let us now turn to Fr\'echet sequence spaces in which $(e_n)_n$ is an unconditional basis. We will see that in this setting bounded completeness of the basis no longer suffices to have all frequently hypercyclic weighted shifts chaotic; see the summary at the end of the paper. The reason is that with any application of the continuity of the weighted shift we potentially lose quality of the seminorm. To be more precise, let us fix an increasing sequence $(\|\cdot\|_m)_{m\geq 1}$ of seminorms defining the topology of $X$. Then continuity of an iterate $B_w^n$, $n\geq 1$, of $B_w$ means that, for any $m\geq 1$, $\|B_w^nx\|_m$ is majorized by a multiple of $\|x\|_q$ but where $q$ may depend on $n$ and $m$; this makes our Banach space proof break down. It turns out that it suffices to impose the property of topologizability on the weighted shifts, see \cite{Bon07}:\\

(T) Any weighted shift $B_w$ on $X$ is topologizable, that is, for any $m\geq 1$ there is some $q(m)\geq 1$ such that for any $n\geq 0$ there is some constant $C_{m,n}>0$ such that, for any $x\in X$,
\begin{equation*}
\|B_w^nx\|_m\leq C_{m,n} \|x\|_{q(m)}.
\end{equation*}
~\medskip

The property (T) is a condition on the space $X$, and it is independent of the particular sequence $(\|\cdot\|_m)_{m\geq 1}$ chosen. We then have the following generalization of Theorem \ref{thm-Banach}.

\begin{theorem}\label{thm-Frechet}
Let $X$ be a Fr\'echet sequence space that satisfies \emph{(T)} and for which $(e_n)_n$ is a boundedly complete unconditional basis. Let $B_w$ be a weighted shift on $X$. Then the following assertions are equivalent:
\begin{enumerate}
\item[\rm (i)] $B_{w}$ is $\UU$-frequently hypercyclic on $X$;
\item[\rm (ii)] $B_{w}$ is frequently hypercyclic on $X$;
\item[\rm (iii)] $B_{w}$ is chaotic on $X$;
\item[\rm (iv)] the series $\sum _{n\geq 0}\frac{1}{w_1\cdots w_n}e_n$ is convergent in $X$. 
\end{enumerate}
\end{theorem}

\begin{proof} First, the unconditionality of the basis implies that, whenever $x=\sum_{n\geq 0} x_n e_n \in X$ and $b=(b_n)_n$ is a bounded sequence of scalars, then $x=\sum_{n\geq 0} b_n x_n e_n \in X$; in addition, for any $m\geq 1$, there exist a constant $K_m>0$ and some $p(m)\geq 1$ such that, for all such $x$ and $b$,
\begin{equation}\label{eq5}
\Big\Vert \sum _{n\geq 0} b_nx_ne_n\Big\Vert_m\leq K_m\|b\|_\infty\Big\Vert \sum _{n\geq 0} x_ne_n\Big\Vert_{p(m)};
\end{equation}
see \cite[Proposition 3.1.3]{AlKa06} in the case of a Banach space, but the proof works as well for Fr\'echet spaces, using \cite[Theorem 3.3.9]{KaGu81}. 

It suffices again to prove that (i) implies (iv). Thus, let $x=(x_n)_n \in X$ be a $\UU$-frequently hypercyclic vector for $B_{w}$ and $m\ge 1$. We fix $s\geq 1$ such that $\|e_0\|_s>0$ and we let
\[u=q(p(m)),\quad r=p(p(u))\quad\text{and}\quad t=\max(p(s),p(p(r)))\]
where $q(\cdot)$ comes from property (T). Since $x$ is $\UU$-frequently hypercyclic and $\|e_0\|_s>0$, the set 
\[
A:=\Big\{n\in \N_0 :\Vert B_{w}^nx-e_0\Vert_t \leq \frac{\Vert e_0\Vert_s}{2K_s}\Big\}
\]
has positive upper density. By \eqref{eq5} we obtain for any $n\in A$
\[
\Big\Vert \frac{1}{v_n}x_{n}e_0-e_0\Big\Vert_s=\Vert w_{1}\cdots w_{n}x_{n}e_0-e_0\Vert_s\leq K_s \Vert B_{w}^nx-e_0\Vert_{p(s)} \leq K_s \Vert B_{w}^nx-e_0\Vert_{t} \leq \frac{\Vert e_0\Vert_s}{2},
\]
where $v$ is the sequence associated to $w$ by \eqref{eq_v}. We continue as in the proof of Theorem~\ref{thm-Banach}, obtaining first that $|\frac{x_n}{v_n}|\geq 1/2$ for any $n\in A$, and then for any integer $M\geq 0$,
\begin{align*}
\Big\Vert\sum_{\substack{m\in A\\n\leq m\leq n+M}}v_{m-n}e_{m-n}\Big\Vert_r &
\leq 2K_r K_{p(r)} \Vert B_w^nx\Vert_{p(p(r))} \leq  2K_r K_{p(r)} (\Vert B_w^nx-e_0\Vert_t+\Vert e_0\Vert_t)\\
&\leq 2K_r K_{p(r)}(\tfrac{\Vert e_0\Vert_s}{2K_s}+\Vert e_0\Vert_t).
\end{align*}
This shows that the family
\[
\Big(\sum_{\substack{m\in A\\n\leq m\leq n+M}}v_{m-n}e_{m-n}\Big)_{n\in A, M\geq 0}
\]
is bounded with respect to the seminorm $\|\cdot\|_r$; note however that $A$ depends on $r$. Writing for $N,M\geq 0$
\[
y_{N,M} = \frac{1}{N+1}\sum_{\substack{n\in A\\n\leq N}}\sum _{\substack{m\in A\\n\leq m\leq n+M}}v_{m-n}e_{m-n},
\]
we obtain that also the family $(y_{N,M})_{N,M\geq 0}$ is bounded with respect to the seminorm $\|\cdot\|_r$.

Let $\delta:=\overline{\text{dens}}(A)$ and $0<\varepsilon<\delta^2$. Then we obtain exactly as in the proof of Lemma \ref{lemma1} by a double application of unconditionality that there is a syndetic set $F\subset \N_0$ so that
\[
C:=\sup_{M\geq 0}\Big\Vert \sum_{k\in F,\,k\leq M}v_ke_{k}\Big\Vert_u\leq \sup_{j\geq 1,M\geq 0}\frac{K_uK_{p(u)}}{\delta^2-\varepsilon} \Vert y_{N_j,M}\Vert_{p(p(u))}=\sup_{j\geq 1,M\geq 0}\frac{K_uK_{p(u)}}{\delta^2-\varepsilon} \Vert y_{N_j,M}\Vert_{r}<\infty.
\]
In particular, there is some $K\geq 0$ such that 
\[
\N_0=\bigcup_{k\leq K} (F-k)
\]
and by property (T) there is a constant $\tilde{C}>0$ such that, for any $x\in X$ and $0\leq n\leq K$,
\[
\|B_w^nx\|_{p(m)}\leq \tilde{C}\|x\|_{q(p(m))}=\tilde{C}\|x\|_{u}.
\]
 
Following now exactly the proof of Lemma \ref{lemma2} we obtain, together with an application of unconditionality, that for any $N\geq 0$
\[
\Big\|\sum_{0\leq n\leq F_N} v_ne_{n}\Big\|_m \leq K_m(K+1)\tilde{C}C,
\]
where $F_N=\max\{n\in F: n\leq N\}$, with $\max\varnothing =-1$. 

Note that $m\geq 1$ is arbitrary here. Since $F_N\to\infty$ as $N\to\infty$, another application of unconditionality shows that the sequence 
\[
\Big(\sum_{0\leq n \leq N} \frac{1}{w_{1}\cdots w_{n}}e_{n}\Big)_{N\geq 0}=\Big(\sum_{0\leq n \leq N} v_ne_{n}\Big)_{N\geq 0}
\]
is bounded in $X$, so that the bounded completeness of the basis implies (iv).
\end{proof}

\begin{remark}\label{rem-T}
(a) Of course, one could just demand that condition (T) holds for a given weighted shift $B_w$ on $X$, and then we could conclude that, for this shift, frequent hypercyclicity implies chaos. But given the simplicity of the characterization of chaos, it would be much easier to show directly that $B_w$ is chaotic.

(b) Remark \ref{rem-genufhc} also applies to Fr\'echet sequence spaces.
\end{remark}

\section{Chaotic and frequently hypercyclic weighted shifts on K\"othe sequence spaces}\label{sec3}

In this section, we aim to apply Theorem \ref{thm-Frechet} to a particularly interesting class of Fr\'echet sequence spaces, the K\"othe sequence spaces, also called K\"othe echelon spaces, see \cite[Chapter 27]{MeVo97} or \cite{BiBo03}.

Let $A=(a_{m,n})_{m\geq 1,n\geq 0}$ be a matrix of strictly positive numbers such that, for all $m\geq 1$, $n\geq 0$,
\[
a_{m,n}\leq a_{m+1,n}.
\]
Such a matrix is called a \textit{K\"othe matrix}.

Let $1\leq p <\infty$. We recall that the \textit{K\"othe sequence space of order $p$} is defined as
\[
\lambda^p(A) = \Big\{ x=(x_n)_{n\geq 0} : \text{for all $m\geq 1$, } \sum_{n\geq 0} |x_n|^pa_{m,n}<\infty\Big\},
\] 
while the \textit{K\"othe sequence spaces of order $0$ and $\infty$} are given by
\[
c_0(A) = \Big\{ x=(x_n)_{n\geq 0} : \text{for all $m\geq 1$, } \lim_{n\to\infty}|x_n|a_{m,n}=0\Big\},
\] 
\[
\lambda^\infty(A) = \Big\{ x=(x_n)_{n\geq 0} : \text{for all $m\geq 1$, } \sup_{n\geq 0}|x_n|a_{m,n}<\infty\Big\}.
\] 
The topologies are respectively induced by the (semi-)norms
\[
\|x\|_m = \Big(\sum_{n\geq 0} |x_n|^pa_{m,n}\Big)^{1/p},\ m\geq 1,\quad \text{for $1\leq p<\infty$},
\]
\[
\|x\|_m = \sup_{n\geq 0} |x_n|a_{m,n},\ m\geq 1,\quad \text{for order $0$ and $\infty$}.
\]

We note that, for each of the spaces $\lambda^p(A)$, $1\leq p<\infty$, and $c_0(A)$, the sequence $(e_n)_n$ is an unconditional basis. On the other hand, the space $\lambda^\infty(A)$ only has $(e_n)_n$ as a basis if it coincides with $c_0(A)$ (which can happen, see Proposition \ref{caracc_0}); hence we will only study weighted shifts on K\"othe sequence spaces of finite order.

\begin{remark}\label{rem-KoBa}
It is easy to see that a K\"othe sequence space $X=\lambda^p(A)$, $1\leq p <\infty$, or $X=c_0(A)$ is normable and hence a Banach space if and only if 
\[
\exists\mu\geq 1,\, \forall m\geq 1\,:\,\sup_{n\geq 0}\frac{a_{m,n}}{a_{\mu,n}}<\infty.
\]
In the first case, $X = \{(x_n)_{n} :  \sum_{n\geq 0} a_{\mu,n}|x_n|^p<\infty\}$ is a weighted $\ell^p$-space, which implies that any frequently hypercyclic weighted shift on $X$ is chaotic. In the second case, $X$ is a weighted $c_0$-space, so that there is a frequently hypercyclic weighted shift on $X$ that is not chaotic.
\end{remark}

Now, in order to apply Theorem \ref{thm-Frechet} we need to ensure that the sequence $(e_n)_n$ is boundedly complete. For the Köthe sequence spaces of order $p\in [1,\infty)$, this is clearly always the case. For the simplest K\"othe sequence space of order $0$, $c_0$, or more generally, when $c_0(A)$ is a Banach space, then the sequence $(e_n)_n$ is not boundedly complete. But this is not so for all K\"othe sequence space of order $0$. Let us assume that $(e_n)_n$ is boundedly complete in  $c_0(A)$, and let $x=(x_n)_n$ be an arbitrary sequence in $\lambda ^{\infty}(A)$. Then the sequence $(\sum _{n=0}^Nx_ne_n)_{N\geq 0}$ is bounded in $c_0(A)$ and thus $x$ belongs to $c_0(A)$. Therefore $c_0(A)=\lambda ^{\infty}(A)$. It is just as easy to see that if $c_0(A)=\lambda ^{\infty}(A)$ then $(e_n)_n$ is boundedly complete in $c_0(A)$. 

Now, the identity $c_0(A)=\lambda ^{\infty}(A)$ can be characterized in terms of the entries of $A$, see \cite[Theorem 27.9]{MeVo97}. Thus we have the following.

\begin{proposition}\label{caracc_0}
Let $A$ be a K\"othe matrix. Then the following assertions are equivalent:
\begin{enumerate}
\item[\rm (i)] $(e_n)_n$ is a boundedly complete basis for $c_0(A)$;
\item[\rm (ii)] $c_0(A)=\lambda^{\infty}(A)$;
\item[\rm (iii)] $A$ satisfies the condition
\begin{equation}\label{BC}
\forall I\subset \N\text{ infinite},\, \forall m\geq 1,\,\exists\mu\geq 1:\,\inf_{n\in I}\frac{a_{m,n}}{a_{\mu,n}}=0.\tag{BC}
\end{equation}

\end{enumerate}
\end{proposition}

We mention in passing that Chapter 27 of Meise and Vogt \cite{MeVo97} contains characterizations of various other interesting properties of K\"othe sequence spaces in terms of the K\"othe matrix~$A$. We single out the following result that will be used later; see \cite[Proposition 27.16]{MeVo97} and use that $\lambda^{p}(A)\subset\lambda^{q}(A)\subset c_0(A)\subset \lambda^\infty (A)$ if $1\leq p\leq q <\infty$.

\begin{proposition}\label{caracNuc}
Let $A$ be a K\"othe matrix. Then the following assertions are equivalent:
\begin{enumerate}
\item[\rm (i)] for some $p\in[1,\infty)$, $c_0(A)=\lambda^{p}(A)$;
\item[\rm (ii)] for some $p\in[1,\infty)$, $\lambda^\infty(A)=\lambda^{p}(A)$;
\item[\rm (iii)] for some  $p\neq q$ in $[1,\infty)$, $\lambda^p(A)=\lambda^{q}(A)$;
\item[\rm (iv)] for all  $p\in[1,\infty)$, $\lambda^p(A)=c_0(A)=\lambda^{\infty}(A)$;
\item[\rm (v)] $A$ satisfies the condition
\begin{equation}\label{N}
\forall m\geq 1,\,\exists\mu\geq 1:\,\sum_{n\geq 0}\frac{a_{m,n}}{a_{\mu,n}}<\infty.\tag{N}
\end{equation}
\end{enumerate}
\end{proposition}

Of course, \eqref{N} implies \eqref{BC}. Note that \eqref{N} also characterizes when $c_0(A)$ (or any $\lambda^p(A)$, $p\in [1,\infty]$) is nuclear, see \cite[Proposition 28.16]{MeVo97}.
 
It is now an easy matter to apply Theorem \ref{thm-Frechet} to (certain) K\"othe sequence spaces and thus obtain conditions under which any frequently hypercyclic weighted shift on such a space is chaotic. However, before stating this result in Subsection \ref{subsec-Koethe-pos}, we will observe a rather unexpected phenomenon: there are some K\"othe sequence spaces that do not support any hypercyclic weighted shift.

\subsection{Existence of (chaotic or frequently) hypercyclic weighted shifts on K\"othe sequence spaces}\label{sec-exist}

Let us pause briefly in order to reflect on the question we are trying to discuss. On the positive side we want to show that, on a given Fr\'echet sequence space, every frequently hypercyclic weighted shift is chaotic. Now it might be that the problem is meaningless because there might not be any frequently hypercyclic weighted shifts on the space. Indeed, for (frequent) hypercyclicity it helps if the weights are big, while the continuity of the weighted shift limits their growth. There is then a trade-off.

Even more, it is known that there are separable infinite-dimensional Banach spaces that do not support any chaotic or any frequently hypercyclic operator, see \cite{BMP01}, \cite{Shk09}. For shifts, we do have a characterization of frequently hypercyclic weighted shifts on an arbitrary Fr\'echet sequence space in which $(e_n)_n$ is an unconditional basis; see \cite[Theorem 6.2]{BoGr18}. However, the conditions are rather involved. So let us approach the problem from above (are there any weighted shifts? are there hypercyclic weighted shifts?) and from below (are there chaotic weighted shifts?).

\subsubsection{Existence of weighted shifts}
The existence of weighted shifts on (most) Banach or Fr\'echet sequence spaces is easily characterized.

\begin{proposition}\label{prop-existshift}
\emph{(a)} Any Banach sequence space that contains all finite sequences admits a weighted shift.

\emph{(b)} Let $X$ be a Fr\'echet sequence space in which $(e_n)_{n\geq 0}$ is a basis. Let $(\|\cdot\|_m)_m$ be an increasing sequence of seminorms that generates the topology of $X$. Then $X$ admits a weighted shift if and only if
\begin{equation}\label{eq6}
\forall m\geq 1,\, \exists \mu\geq 1,\, \forall n\geq 0:\, \|e_n\|_m \neq 0 \Longrightarrow \|e_{n+1}\|_{\mu}\neq 0.
\end{equation}
\end{proposition}

\begin{proof}
We first show sufficiency of \eqref{eq6} in (b). Let, for any $m\geq 1$, $\mu_m=:\mu\geq 1$ be given by this condition. Since $(e_n)_n$ is a basis, the Banach-Steinhaus theorem implies that there are $\nu_m\geq 1$ and $C_m>0$ such that, for all $x\in X$, $n\geq 0$ and $m\geq 1$,
\begin{equation}\label{eq7}
\|x_ne_n\|_{\mu_m}\leq C_m\|x\|_{\nu_m}.
\end{equation}
Let 
\[
a_n=\max_{m\leq n}  C_m\frac{\|e_n\|_m}{\|e_{n+1}\|_{\mu_m}},\quad n\geq 0;
\]
note that this is well-defined by \eqref{eq6} if we let $\frac{0}{0}=0$. 

Now let $w=(w_n)_{n\geq 1}$ be a sequence of non-zero scalars such that
\[
\sum_{n\geq 0} |w_{n+1}| a_n <\infty.
\]
Then we have that, for any $m\geq 1$ and $x\in X$,
\begin{align*}
\sum_{n\geq 0} \|w_{n+1}x_{n+1}e_n\|_m &\leq \sum_{n\geq 0} |w_{n+1}| \frac{\|e_n\|_m}{\|e_{n+1}\|_{\mu_m}}  \|x_{n+1}e_{n+1}\|_{\mu_m}\\
& \leq \Big(\sum_{n\geq 0} |w_{n+1}| C_m \frac{\|e_n\|_m}{\|e_{n+1}\|_{\mu_m}}\Big)  \|x\|_{\nu_m}\\
& \leq \Big(\sum_{n=0}^{m-1} |w_{n+1}| C_m \frac{\|e_n\|_m}{\|e_{n+1}\|_{\mu_m}}  + \sum_{n\geq m} |w_{n+1}| a_n\Big)  \|x\|_{\nu_m} <\infty,
\end{align*}
so that $B_w$ defines a weighted shift on $X$. 

The necessity of \eqref{eq6} follows directly from the continuity of any given weighted shift on $X$.

For (a) note that, in the Banach space case, \eqref{eq6} is automatic. Moreover, inequality \eqref{eq7} holds even without the basis assumption when one allows the constant $C_m$ to depend on $n$ since the projections $x\to x_ne_n$, $n\geq 1$, are continuous; the proof can then be continued as before.
\end{proof}

In particular, any Köthe sequence space of finite order admits a weighted shift.

\begin{example}\label{ex-FreEx}
We give here an example of a Fr\'echet sequence space in which $(e_n)_n$ is a basis but that admits no weighted shift. Indeed, let 
\[
X=\Big\{ (x_n)_n : \sum_{n\geq 0} |x_{2n}|<\infty\Big\}.
\]
This turns into a Fr\'echet sequence space under the seminorms
\[
\|x\|_m = \max_{0\leq n\leq m}|x_{2n+1}| + \sum_{n\geq 0} |x_{2n}|,\quad m\geq 1.
\]
Then $X$ does not admit any weighted shift because no weight can turn all sequences into $\ell^1$-sequences.
\end{example}

\subsubsection{Existence of hypercyclic weighted shifts}

For a general Banach or Fr\'echet sequence space, the existence of hypercyclic weighted shifts seems to be a complicated matter. Thus we restrict our attention to Köthe sequence spaces.

Note that a weighted shift $B_w$ defines an operator on a K\"othe sequence space $\lambda^p(A)$ or $c_0(A)$ if and only if 
\begin{equation*}
\forall m\geq 1,\,\exists\mu\geq 1,C>0,\,\forall n\geq 1:\,
\begin{cases}
|w_n|^p a_{m,n-1}\leq C a_{\mu,n}, &\text{on $\lambda^p(A)$},\\
|w_n| a_{m,n-1}\leq C a_{\mu,n}, &\text{on $c_0(A)$}.
\end{cases}
\end{equation*}
For later use we note here that, in terms of the associated sequence $v$ (see \eqref{eq_v} and \eqref{eq_w}), the conditions can be written as
\begin{equation}\label{eq8v}
\forall m\geq 1,\,\exists\mu\geq 1,C>0,\,\forall n\geq 1:\,
\begin{cases}
|v_{n-1}|^p a_{m,n-1}\leq C |v_{n}|^pa_{\mu,n}, &\text{on $\lambda^p(A)$},\\
|v_{n-1}| a_{m,n-1}\leq C |v_{n}|a_{\mu,n}, &\text{on $c_0(A)$}.
\end{cases}
\end{equation}

\begin{proposition}\label{prop-exhcko}
Let $X=\lambda^p(A)$, $1\leq p<\infty$, or $X=c_0(A)$ be a K\"othe sequence space. Then there exists a hypercyclic weighted shift on $X$ if and only if there exist a strictly increasing sequence $(m_k)_k$ of positive integers and a positive sequence $(C_k)_k$ such that
\[
\forall n\geq 1,\,\nu_n:=\inf_{k\geq 1} C_k \frac{a_{m_{k+1},n}}{a_{m_k,n-1}}>0
\]
and
\[  
\forall m\geq 1,\,\liminf_{N\to\infty} \frac{1}{\prod_{n=1}^{N}\nu_n}a_{m,N}=0.
\]
\end{proposition}

\begin{proof}
Let $B_{w}$ be a hypercyclic weighted shift on $X$. In this proof, let $p=1$ for the space $c_0(A)$. By continuity, there exist a strictly increasing sequence $(m_k)_k$ of positive integers and a positive sequence $(C_k)_k$ such that, for every $n\ge 1$ and $k\ge 1$, we have
\[
|w_n|^p a_{m_k,n-1}\le C_k a_{m_{k+1},n}
\]
hence, for any $n\geq 1$,
\[
\nu_n:=\inf_{k\geq 1} C_k \frac{a_{m_{k+1},n}}{a_{m_k,n-1}}\geq |w_n|^p>0.
\]
Moreover, since $B_w$ is hypercyclic, we have that, for every $m\geq 1$,
\begin{equation*}
\liminf_{N\to\infty} \frac{1}{\prod_{n=1}^{N}\nu_n}a_{m,N}\leq \liminf_{N\to\infty} \frac{1}{\prod_{n=1}^{N}|w_n|^p}a_{m,N}=0,
\end{equation*}
see \cite[Theorem 4.8]{GrPe11}. This shows necessity.

The sufficiency is easily seen by considering the weighted shift $B_w$ with $w_n=\nu_n^{1/p}$, $n\geq 1$, and using the fact that $a_{m,n}\leq a_{m+1,n}$ for all $m\geq 1$, $n\geq 0$.
\end{proof}

Under an additional assumption on the K\"othe matrix $A$, the characterising condition can be simplified.

\begin{corollary}\label{cor-exhcko}
Let $X=\lambda^p(A)$, $1\leq p<\infty$, or $X=c_0(A)$ be a K\"othe sequence space. Suppose that
\begin{equation}\label{eq10} 
\forall m,j\ge 1,\, \exists J\ge 1, C>0,\, \forall n\geq 0:\,\frac{a_{m,n+1}}{a_{1,n}}\leq C \frac{a_{J,n+1}}{a_{j,n}}.
\end{equation} 
Then there exists a hypercyclic weighted shift on $X$ if and only if 
\begin{equation}\label{eq10bis} 
\exists \mu\ge 1,\, C>0,\, \forall m\ge 1:\, \liminf_{N\to\infty} \frac{\prod_{n=1}^{N}a_{1,n-1}}{C^N\prod_{n=1}^{N}a_{\mu,n}}a_{m,N}=0.
\end{equation} 
\end{corollary}

\begin{proof}
The condition is even necessary without assumption \eqref{eq10} on $A$. Indeed, if $X$ supports a hypercyclic weighted shift then by the previous proposition there is a strictly positive sequence $(\nu_n)_n$ as well as $m_1$, $\mu:=m_2\geq 1$ and $C>0$ such that
\[
\nu_n\leq  C \frac{a_{\mu,n}}{a_{m_1,n-1}}\leq C\frac{a_{\mu,n}}{a_{1,n-1}},\ n\geq 1, \quad\text{and}\quad \liminf_{N\to\infty} \frac{1}{\prod_{n=1}^{N}\nu_n}a_{m,N}=0,\ m\geq 1.
\]
This implies \eqref{eq10bis}.

On the other hand, suppose that \eqref{eq10bis} holds with some $\mu\geq 1$ and $C>0$. Then one obtains by \eqref{eq10} a strictly increasing sequence $(m_k)_k$ of positive integers and a positive sequence $(C'_k)_k$ such that, for all $k\geq 1$ and $n\geq 0$,
\[
\frac{a_{\mu,n+1}}{a_{1,n}}\leq C'_k \frac{a_{m_{k+1},n+1}}{a_{m_k,n}}.
\]
In particular, setting $C_k=CC'_k$, we deduce that for any $n\geq 1$,
\[
\nu_n:=\inf_{k\geq 1} C_k\frac{a_{m_{k+1},n}}{a_{m_k,n-1}}\ge C\frac{a_{\mu,n}}{a_{1,n-1}}>0
\]
and that for every $m\ge 1$,
\[
\liminf_{N\to\infty} \frac{1}{\prod_{n=1}^{N}\nu_n}a_{m,N}\le \liminf_{N\to\infty} \frac{1}{C^N\prod_{n=1}^{N}\frac{a_{\mu,n}}{a_{1,n-1}}}a_{m,N} =0.
\]
\end{proof}

A Köthe sequence space without hypercyclic weighted shifts can arise quite naturally.

\begin{example}\label{ex-nohc}
Let $X$ be the sequence space given by
\[
X= \Big\{ x=(x_n)_{n\geq 0} : \sum_{n\geq 0} |x_n| \rho^{2^n} <\infty\text{ for all }\rho>0\Big\}.
\]
This can be considered as the space of entire functions with certain lacunary power series. Then $X$ is the K\"othe sequence space $\lambda^1(A)$ with $a_{m,n} = m^{2^n}$, $m\geq 1$, $n\geq 0$. It is easily verified that \eqref{eq10} holds but \eqref{eq10bis} does not. Thus, $X$ supports no hypercyclic weighted shift.
\end{example}

This example will be generalized in Example \ref{ex infinite type}(a).

\subsubsection{Existence of chaotic weighted shifts}
This case can be treated exactly as the existence of hypercyclic shifts. Note, however, that the characterizing condition of chaos, $\Big(\frac{1}{\prod_{n=1}^{N}w_n}\Big)_{N\geq 0}\in X$ (see condition \eqref{eq0}) depends on the order of the Köthe sequence space.

\begin{proposition}\label{prop-exchako}
Let $X=\lambda^p(A)$, $1\leq p<\infty$, or $X=c_0(A)$ be a K\"othe sequence space. Then there exists a chaotic weighted shift on $X$ if and only if there exist a strictly increasing sequence $(m_k)_k$ of positive integers and a positive sequence $(C_k)_k$ such that
\[
\forall n\geq 1,\,\nu_n:=\inf_{k\geq 1} C_k \frac{a_{m_{k+1},n}}{a_{m_k,n-1}}>0\quad\text{and}\quad 
\Big(\frac{1}{\prod_{n=1}^{N}\nu_n ^{1/p}}\Big)_{N\geq 0}\in X,
\]
where $p=1$ if $X=c_0(A)$.
\end{proposition}

\begin{corollary}\label{cor-exchako}
Let $X=\lambda^p(A)$, $1\leq p<\infty$, or $X=c_0(A)$ be a K\"othe sequence space such that \eqref{eq10} holds.
Then there exists a chaotic weighted shift on $X$ if and only if
\[
\exists \mu\ge 1,\,  C>0:\,\Big(\frac{\prod_{n=1}^{N}a_{1,n-1} ^{1/p}}{C^N\prod_{n=1}^{N}a_{\mu,n} ^{1/p}}\Big)_{N\geq 0} \in X,
\]
where $p=1$ if $X=c_0(A)$.
\end{corollary}

\subsection{When frequent hypercyclicity implies chaos on K\"othe sequence spaces}\label{subsec-Koethe-pos}

Let us spell out Theorem \ref{thm-Frechet} in the case of K\"othe sequence spaces. 

\begin{theorem}\label{thm-kothe}
Let $X$ be a Köthe sequence space $\lambda^p(A)$, or a Köthe sequence space $c_0(A)$ that satisfies \eqref{BC}. Suppose that the following condition holds: 

\emph{(TK)} for any strictly positive sequence $v=(v_n)_n$, if
\begin{equation*}
\forall m\geq 1,\,\exists\mu\geq 1, C>0,\,\forall n\geq 1:\, v_{n-1}a_{m,n-1}\leq C v_{n} a_{\mu,n},
\end{equation*}
then 
\begin{equation}\label{eq13}
\forall m\geq 1,\,\exists\mu\geq 1,\,\forall j\geq 1,\,\exists\  C_j>0,\,\forall n\geq j:\,v_{n-j} a_{m,n-j} \leq C_j v_{n} a_{\mu,n}.
\end{equation}

Let $B_w$ be a weighted shift on $X$. Then the following assertions are equivalent:
\begin{enumerate}
\item[\rm (i)] $B_{w}$ is $\UU$-frequently hypercyclic on $X$;
\item[\rm (ii)] $B_{w}$ is frequently hypercyclic on $X$;
\item[\rm (iii)] $B_{w}$ is chaotic on $X$;
\item[\rm (iv)] the series $\sum _{n\geq 0}\frac{1}{w_1\cdots w_n}e_n$ is convergent in $X$.
\end{enumerate}
\end{theorem}

Indeed, since $B_w$ being an operator on $X$ can be expressed by condition \eqref{eq8v}, with a similar condition for $B_w^j$, it is straightforward to see that condition (T) for the space $X$ is equivalent to condition (TK) for the K\"othe matrix $A$.

In the case of Köthe sequence spaces $\lambda^p(A)$ we have already examples to which the theorem is applicable: these are given by the classical $\ell^p$-spaces since they are Banach spaces and thus satisfy (T) trivially. But there also are non-Banach examples and even examples of order $0$.

\begin{proposition}\label{propHD}
Every frequently hypercyclic weighted shift on $H(\D)$ is chaotic.
\end{proposition}

\begin{proof}
The space $H(\D)$ of holomorphic functions on the unit disk $\D$ can be identified with the Köthe sequence spaces $c_0(A)$ or $\lambda^p(A)$ for any $1\le p<\infty$ by letting $a_{m,n}=\frac{1}{\left(1+\frac{1}{m}\right)^{n}}$. These spaces coincide since $A$ satisfies (N), and it suffices to show that $A$ satisfies (TK). Let $(v_n)_n$ be a strictly positive sequence such that 
\begin{equation*}
\forall m\geq 1,\,\exists\mu\geq 1, C>0,\,\forall n\geq 1:\, v_{n-1}a_{m,n-1}\leq C v_{n} a_{\mu,n}.
\end{equation*}
We can first show that $\limsup_{n} \big(\frac{v_{n-1}}{v_n}\big)^{\frac{1}{n}}\le 1$. Indeed, given $m\geq 1$, $\mu\geq 1$, $C>0$ such that, for every $n\geq 1$, $v_{n-1}a_{m,n-1}\leq C v_{n} a_{\mu,n}$, we get
\[
\frac{v_{n-1}}{v_n}\le C \frac{a_{\mu,n}}{a_{m,n-1}}=\frac{C}{1+\frac{1}{m}}\Big(\frac{1+\frac{1}{m}}{1+\frac{1}{\mu}}\Big)^{n}
\]
and thus $\limsup_{n} \big(\frac{v_{n-1}}{v_n}\big)^{\frac{1}{n}}\le \frac{1+\frac{1}{m}}{1+\frac{1}{\mu}}\le 1+\frac{1}{m}$. Since this inequality holds for every $m\geq 1$, we deduce that $\limsup_{n} \big(\frac{v_{n-1}}{v_n}\big)^{\frac{1}{n}}\le 1$. 

We now show that if  $\limsup_{n} \big(\frac{v_{n-1}}{v_n}\big)^{\frac{1}{n}}\le 1$ then \eqref{eq13} holds.
Let $m\geq 1$, $\mu=m+1$, $j\geq 1$. We need to find a constant $C_j>0$ such that, for every $n\ge j$,
\[
\frac{v_{n-j}}{v_n}\le \frac{C_j}{(1+\frac{1}{m})^j} \Big(\frac{1+\frac{1}{m}}{1+\frac{1}{m+1}}\Big)^{n}.
\]
We remark that such a constant will exist if $\limsup_n \big(\frac{v_{n-j}}{v_n}\big)^{\frac{1}{n}}\le 1$, and since
\[
\limsup_n \Big(\frac{v_{n-j}}{v_n}\Big)^{\frac{1}{n}}\le \prod_{l=0}^{j-1} \limsup_n \Big(\frac{v_{n-l-1}}{v_{n-l}}\Big)^{\frac{1}{n}}
\le \prod_{l=0}^{j-1} \limsup_n \Big(\max\Big(1,\frac{v_{n-l-1}}{v_{n-l}}\Big)\Big)^{\frac{1}{n-l}}\le 1,
\]
we get the desired result.
\end{proof}

Note that the above result is not vacuous since $H(\D)$ supports chaotic weighted shifts such as the differentiation operator \cite{Sha98}. Examples of K\"othe matrices $A$ to which Theorem~\ref{thm-kothe} applies and where all the spaces $c_0(A)$ and $\lambda^p(A)$, $1\leq p<\infty$, are different can also be found (see Example \ref{ex-nonN}(a)).

\subsection{When frequent hypercyclicity does not imply chaos on K\"othe sequence spaces}\label{subsec-Koethe-neg}

We now show that some K\"othe sequence spaces, even of order $p\in[1,\infty)$, support frequently hypercyclic, non-chaotic weighted shifts.

For this we recall the following result; see \cite[Theorem 6.2]{BoGr18}. If $X$ is a Fr\'echet sequence space in which $(e_n)_n$ is an unconditional basis, then a weighted shift $B_w$ on $X$ is frequently hypercyclic if and only if there exist a sequence $(\varepsilon_r)_{r\geq 1}$ of positive numbers with $\varepsilon_r\to 0$ as $r\to\infty$ and a sequence $(A_r)_{r\geq 1}$ of pairwise disjoint sets of positive lower density such that
\begin{itemize}
\item[\rm (i)] for any $r\geq 1$,
\[
\sum_{n\in A_r} v_{n+r}e_{n+r} \quad \text{converges in $X$};
\]
\item[\rm (ii)] for any $r,s\geq 1$, any $m\in A_s$, and any $j=0,\ldots,r$,
\[
\Big\|\sum _{\substack{n\in A_r\\n>m}} v_{n-m+j}e_{n-m+j}\Big\|_s <\min(\varepsilon_r,\varepsilon_s),
\]
\end{itemize}
where $v$ is the sequence associated to $w$ via \eqref{eq_v}. 

Recall also that, by \eqref{eq0}, $B_w$ is chaotic if and only if 
\begin{equation}\label{eq14}
\sum_{n\geq 0} v_ne_n \quad \text{converges in $X$}.
\end{equation}

For the existence of a non-chaotic frequently hypercyclic weighted shift we will modify the construction of such a shift in $c_0$ as given in \cite[Theorem 7.3]{BoGr18}. The following was shown in the proof there.

\begin{lemma}[\cite{BoGr18}]\label{r-constr}
There exists a strictly increasing sequence $(N_k)_{k\geq 0}$ of non-negative integers with $N_0=0$ and pairwise disjoint sets $A_r$, $r\geq 1$, of positive lower density with the following properties:

\emph{(a)} For $r\geq 1$, if $n\in A_r$, then for any $k\geq 0$,
\[
N_{k}\leq n <N_{k+1} \Longrightarrow N_{k}+k \leq n < N_{k+1}-r.
\]

\emph{(b)} For $r,s\geq 1$, if $n\in A_r$, $m\in A_s$, $n> m$, then for any $k\geq 0$,
\[
N_{k}\leq n-m < N_{k+1} \Longrightarrow N_{k}+\max(r,s) \leq n-m < N_{k+1}-\max(r,s).
\]
\end{lemma}

A frequently hypercyclic, non-chaotic weighted shift on $c_0$ was then given by the weight $w$ that is associated, via \eqref{eq_w}, to the following sequence $v=(v_n)_n$:
\[
v_n= \frac{1}{2^{n-N_k}} \quad\text{for $N_k\leq n < N_{k+1}$}, k\geq 0.
\]

For general spaces it seems natural to consider 
\[
v_n=\frac{1}{b_{n-N_k,n}} \quad\text{for $N_k\leq n < N_{k+1}$}, k\geq 0,
\]
for suitable numbers $b_{m,n}$. This leads us to the following condition for Köthe sequence spaces $X$; note that part $(\gamma)$ depends on which space $\lambda^p(A)$ or $c_0(A)$ one considers.\\

(B) There is a strictly positive upper triangular matrix
\[
B=(b_{m,n})_{m\geq 0, n\geq m}
\]
with increasing columns such that the following properties hold:
\[
\forall m\geq 1,\, \exists\mu\geq 1, C>0,\, \forall j\geq 0, n\geq j:\,\frac{a_{m,n}}{a_{\mu,n+1}}\leq C \frac{b_{j,n}}{b_{j+1,n+1}};\tag{$\alpha$}
\]	
\[
 \exists m\geq 1:\, \inf_{n\geq 0} \frac{a_{m,n}}{b_{0,n}} >0;\tag{$\beta$}
\]
and either $(\gamma)$, where
\[
\exists (j_m)_{m\geq 1} \nearrow\,:\begin{cases}  
\forall m\geq 1 : \,\displaystyle\sum_{n\geq j_m}\frac{a_{m,n}}{b_{j_m,n}}<\infty, & \text{if $X=\lambda^p(A)$},\tag{$\gamma$}\\
\displaystyle\lim_{m\to\infty}\limsup_{n\to\infty} \frac{a_{m,n}}{b_{j_m,n}}=0, & \text{if $X=c_0(A)$};
\end{cases}
\]
or $(\widetilde{\gamma})$, where
\begin{equation*}\tag{$\widetilde{\gamma}$}
\begin{cases}
\displaystyle\forall m\geq 1,\, \exists N\geq 0,\, \rho <1,\, \forall j\geq 0,\, n\geq  \max(N,j):\frac{a_{m,n+1}}{a_{m,n}}\leq \rho \frac{b_{j+1,n+1}}{b_{j,n}};\\[2ex]
\displaystyle\exists (n_k)_{k\geq 1} \nearrow, (d_k)_{k\geq 1} \nearrow,\,  \forall m\geq 1:\, \lim_{k\to\infty} \sup_{n\geq \max(d_k,n_k)}\frac{a_{m,n}}{b_{d_k,n}}= 0;\\[2ex]
\displaystyle\exists (j_m)_{m\geq 1} \nearrow,\,  \forall m\geq 1:\, \lim_{n\to\infty} \frac{a_{m,n}}{b_{j_m,n}}=0.
\end{cases}
\end{equation*}

Here, $(j_m)_{m}$, $(n_k)_{k}$ and $(d_k)_{k}$ are strictly increasing sequences of positive integers. Also note that the $n$-th column of the matrix $B$ has (only) $n+1$ entries, and these are supposed to be increasing.

We have introduced the alternative and rather technical condition $(\widetilde{\gamma})$ in order to deal with some interesting spaces that do not satisfy $(\gamma)$. One example is given by the power series space of order $p\geq 1$ and infinite type with $\alpha=(\log(\log(n+3))_n$, see Theorem \ref{thm-pssinfbis}. Note, however, that the third condition in $(\widetilde{\gamma})$ implies $(\gamma)$ when $X=c_0(A)$ so that, for these spaces, condition $(\widetilde{\gamma})$ is of no interest.

\begin{theorem}\label{thm-koefhcnotch}
Let $X=\lambda^p(A)$, $1\leq p<\infty$, or $X=c_0(A)$ be a K\"othe sequence space. If condition \emph{(B)} holds then $X$ supports a chaotic weighted shift, and there exists a frequently hypercyclic weighted shift that is not chaotic.
\end{theorem}

\begin{proof}
We start by showing that under condition (B) the hypotheses of Proposition~\ref{prop-exchako} are satisfied. By $(\alpha)$, there exist a strictly increasing sequence $(m_k)_{k\ge 1}$ of positive integers and a positive sequence $(C_k)_{k\ge 1}$ such that, for every $k\ge 1$ and $n\ge j\ge 0$,
\[
\frac{a_{m_k,n}}{a_{m_{k+1},n+1}}\le C_k \frac{b_{j,n}}{b_{j+1,n+1}}.
\]
Therefore, we have for every $n\ge 1$,
\[
\nu_n:=\inf_{k\ge 1} C_k \frac{a_{m_{k+1},n}}{a_{m_k,n-1}}\ge \max_{j\le n-1}\frac{b_{j+1,n}}{b_{j,n-1}}>0.
\]

It remains to show that if $(\gamma)$ or $(\tilde{\gamma})$ holds then $\big(\frac{1}{\prod_{n=1}^N\nu_n^{1/p}}\big)_{N\ge 0}\in X$, where $p=1$ for the space $X=c_0(A)$. First, let $(j_m)_m$ be a strictly increasing sequence of positive integers. We then deduce that, for every $m\ge 1$ and every $N\geq j_m$, we have
\begin{equation}\label{eq-new}
\begin{split}
\frac{1}{\prod_{n=1}^N\nu_n}a_{m,N}&\le \frac{1}{\prod_{n=1}^{j_m-1}\frac{b_{n,n}}{b_{n-1,n-1}}}  \frac{1}{\prod_{n=j_m}^{N}\frac{b_{j_m,n}}{b_{j_m-1,n-1}}} a_{m,N}\\
&=  \frac{1}{\frac{b_{j_m-1,j_m-1}}{b_{0,0}}}  \frac{b_{j_m-1,j_m-1}\prod_{n=j_m}^{N-1} b_{j_m-1,n}}{\prod_{n=j_m}^{N-1}b_{j_m,n}} \frac{a_{m,N}}{b_{j_m,N}}\\
&\le b_{0,0}\frac{a_{m,N}}{b_{j_m,N}}
\end{split}
\end{equation}
since $B$ has increasing columns. 

If $(\gamma)$ holds, then \eqref{eq-new} implies that $\big(\frac{1}{\prod_{n=1}^N\nu_n^{1/p}}\big)_{N\ge 0}\in X$; note that a sequence $x=(x_n)_n$ belongs to $c_0(A)$ if and only if there is an increasing sequence $(k_m)_m$ of positive integers such that $\sup_{k\geq k_m}|x_k|a_{m,k} \to 0$ as $m\to\infty$.

On the other hand, if $(\tilde{\gamma})$ holds then for every $m\ge 1$ there exist $J_m\ge 0$ and $\rho_m<1$ such that, for every $n\ge J_m$, 
$\frac{a_{m,n+1}}{a_{m,n}}\le \rho_m\frac{b_{1,n+1}}{b_{0,n}}$. We then deduce that, for every $N> J_m$,
\begin{align*}
\frac{1}{\prod_{n=1}^N\nu_n}a_{m,N}&\le  \Big(\prod_{n=1}^N\frac{b_{0,n-1}}{b_{1,n}}\Big)a_{m,N}\\
&\le \Big(\prod_{n=1}^{J_m}\frac{b_{0,n-1}}{b_{1,n}}\Big)\Big(\prod_{n=J_m+1}^N\rho_m\frac{a_{m,n-1}}{a_{m,n}}\Big)a_{m,N}\\
&=\Big(\prod_{n=1}^{J_m}\frac{b_{0,n-1}}{b_{1,n}}\Big)a_{m,J_m}\rho_m^{N-J_m}.
\end{align*}
Since $\rho_m<1$, it follows again that $\big(\frac{1}{\prod_{n=1}^N\nu_n^{1/p}}\big)_{N\ge 0}\in X$.

We can thus deduce in each case from Proposition~\ref{prop-exchako} that there exists a chaotic weighted shift on $X$.\\

We now show that under condition (B), the space $X$ also supports a weighted shift which is frequently hypercyclic but not chaotic. We first consider a K\"othe sequence space $\lambda^p(A)$. 

Let $B=(b_{m,n})_{m,n}$ be a strictly positive upper triangular matrix that satisfies condition (B), where we begin by the variant $(\alpha)$, $(\beta)$ and $(\gamma)$. Let $(N_k)_{k\geq 0}$ be a strictly increasing sequence of non-negative integers with $N_0=0$ and $A_r$, $r\geq 1$, pairwise disjoint sets of positive lower density such that (a) and (b) of Lemma \ref{r-constr} hold. 

We then define a weighted shift $B_w$ with weight $w$ associated, via \eqref{eq_w}, to the sequence $v$ given by
\begin{equation}\label{eq13b}
v_n=\frac{1}{b_{n-N_k,n}^{1/p}},\quad N_k\leq n < N_{k+1}, k\geq 0;
\end{equation}
note the power $1/p$ that is required by the form of the seminorms in $\lambda^p(A)$.

Then we have for $N_k\leq n < N_{k+1}-1$, $k\geq 0$,
\[
\Big(\frac{v_{n+1}}{v_n}\Big)^p = \frac{b_{n-N_k,n}}{b_{n-N_k+1,n+1}},
\]
while for $n=N_{k+1}-1$, $k\geq 0$, we have that
\[
\Big(\frac{v_{n+1}}{v_n}\Big)^p  = \frac{b_{n-N_k,n}}{b_{0,n+1}} \geq \frac{b_{n-N_k,n}}{b_{n-N_k+1,n+1}},
\]
where we have used the fact that the columns of the matrix $B$ are increasing. Thus, in view of \eqref{eq8v}, condition ($\alpha$)
implies that $B_w$ is an operator on $\lambda^p(A)$. 

Next, let $m\geq 1$ be such that ($\beta$) holds. Then
\[
\inf_{k\geq 0} v_{N_k}^p a_{m,N_k} = \inf_{k\geq 0} \frac{a_{m,N_k}}{b_{0,N_k}}>0,
\]
which,  by \eqref{eq14}, implies that $B_w$ is not chaotic on $\lambda^p(A)$.

It remains to show that $B_w$ is frequently hypercyclic on $\lambda^p(A)$. It follows from $(\gamma)$ that for any $m\geq 1$ there are $j_m\geq 1$ and $n_m\geq j_m$ such that
\begin{equation}\label{bwkoe5}
\sum_{n\geq n_m}  \frac{a_{m,n}}{b_{j_m,n}}< \frac{1}{2^m},\tag{$\gamma'$}
\end{equation}
where we may assume that $(j_m)_m$ and $(n_m)_m$ are strictly increasing.

We will then consider the sets
\[
A_r'= A_{n_r},\quad r\geq 1.
\]
It suffices to show that conditions (i) and (ii) in the characterization of frequent hypercyclicity stated above hold for the sequence $(A_r')_{r\geq 1}$ of pairwise disjoint sets of positive lower density. 

Thus, let $r\geq 1$ and $m\geq 1$. Using condition (a) of Lemma \ref{r-constr} we see that
\begin{align*}
\sum_{n\in A_r'} v_{n+r}^p a_{m,n+r} = \sum_{k\geq 0} \sum_{\substack{n\in A_r'\\N_k\leq n< N_{k+1}}}  \frac{a_{m,n+r}}{b_{n+r-N_k,n+r}}.
\end{align*}
Also by (a) we have that if $n\in A_r'$, $N_k\leq n< N_{k+1}$ then $n\geq N_k+k$, so that for all but finitely many $k$ we have that $n+r-N_k\geq j_m$, hence $b_{{n+r-N_k},n+r}\geq b_{j_m,n+r}$, which together with ($\gamma'$) implies (i).

Now, let $r,s\geq 1$, $m\in A'_s$, and $j=0,\ldots,r$. Using condition (b) of Lemma \ref{r-constr} we see that
\begin{align*}
\sum_{\substack{n\in A'_r\\n>m}} v_{n-m+j}^p a_{s,{n-m+j}} &= \sum_{k\geq 0} \sum_{\substack{n\in A'_r, n>m\\N_k\leq n-m < N_{k+1}}} 
\frac{a_{s,{n-m+j}}}{b_{n-m+j-N_k,{n-m+j}}}\\
&\leq \sum_{k\geq 0} \sum_{\substack{n\in A'_r, n>m\\N_k\leq n-m < N_{k+1}}} \frac{a_{\max(r,s),{n-m+j}}}{b_{j_{\max(r,s)},{n-m+j}}},
\end{align*}
where we have used the fact that if $n\in A_r'=A_{n_r}$, $m\in A_s'=A_{n_s}$, $n>m$, $N_k\leq n-m< N_{k+1}$ for $k\geq 0$, and $j=0,\ldots,r$ then, by (b), $n-m+j- N_k\geq\max(n_r,n_s)\geq \max(j_r,j_s)=j_{\max(r,s)}$. Now, for these $n,m,j$ we have that $n-m+j \geq \max(n_r,n_s)=n_{\max(r,s)}$. Hence, with condition ($\gamma'$), we have that
\begin{align*}
\sum_{\substack{n\in A'_r\\n>m}} v_{n-m+j}^p a_{s,{n-m+j}} \leq \sum_{\nu\geq n_{\max(r,s)}}  \frac{a_{\max(r,s),{\nu}}}{b_{j_{\max(r,s)},{\nu}}}< \frac{1}{2^{\max(r,s)}}= \min(\tfrac{1}{2^r},\tfrac{1}{2^s}).
\end{align*}
This shows (ii), which completes the proof for $\lambda^p(A)$ under condition (B) with $(\alpha)$, $(\beta)$ and $(\gamma)$.\\ 

The proof for $c_0(A)$ is very similar if one takes $p=1$ in \eqref{eq13b} and notes again that $x=(x_n)_n\in c_0(A)$ if and only if there is an increasing sequence $(k_m)_m$ of positive integers such that $\sup_{k\geq k_m}|x_k|a_{m,k} \to 0$ as $m\to\infty$.\\ 

Next suppose that (B) holds with $(\alpha)$, $(\beta)$ and $(\widetilde{\gamma})$, where it suffices to consider the case when $X=\lambda^p(A)$; see the remark before the statement of the theorem. Choose again $(N_k)_{k\geq 0}$ and $(A_r)_{r\geq 1}$ according to Lemma \ref{r-constr}. Throughout the remainder of the proof, for $m\geq 1$, let $\nu_m=N$ and $\rho_m=\rho$ be given by the first condition in $(\widetilde{\gamma})$. We may assume that $(\nu_m)_m$ is strictly increasing. 

Also, let $(d_k)_k$ and $(j_m)_m$ be the sequences of integers given by the second and third conditions in $(\widetilde{\gamma})$. Then we can find inductively strictly increasing sequences $(l_k)_{k\geq 1}$, $(l'_k)_{k\geq 1}$ of strictly positive integers such that, if $k\geq 1$ and $1\leq m\leq k$, 
\begin{equation}\label{eq90}
\frac{a_{m,N_{l'_k}+d_{l_k}}}{b_{d_{l_k},N_{l'_k}+d_{l_k}}}\leq \frac{1}{2^k}
\end{equation}
and
\begin{equation}\label{eq91}
\frac{a_{m,N_{l'_k}+j_m}}{b_{j_m,N_{l'_k}+j_m}}\leq (1-\rho_m)\frac{1}{2^k},
\end{equation}
and we set $l_0=l'_0=d_0=0$. We may also assume that
\[
l'_k\geq d_{l_k}, \quad k\geq 1.
\]
We then set
\[
N'_k=N_{l'_k}, d'_k = d_{l_k}, \quad k\geq 0.
\]
It follows from (a) and (b) of Lemma \ref{r-constr} that we have

(a') For $r\geq 1$, if $n\in A_r$, then for any $k\geq 0$,
\[
N'_{k}\leq n <N'_{k+1} \Longrightarrow N'_{k}+d'_k \leq n < N'_{k+1}-r.
\]

(b') For $r,s\geq 1$, if $n\in A_r$, $m\in A_s$, $n> m$, then for any $k\geq 0$,
\[
N'_{k}\leq n-m < N'_{k+1} \Longrightarrow N'_{k}+\max(r,s) \leq n-m < N'_{k+1}-\max(r,s).
\]

We now define a weighted shift $B_w$ via a sequence $v$ exactly as in \eqref{eq13b}, however with $(N_k)_k$ replaced by $(N'_k)_k$. Conditions $(\alpha)$ and $(\beta)$ imply again that $B_w$ is an operator on $X$ that is not chaotic. It remains to show that conditions (i) and (ii) in the characterization of frequent hypercyclicity hold.

Let $r\geq 1$, $m\geq 1$. It follows with (a') that, for $k\geq 0$,
\begin{equation}\label{eq16}
\sum_{\substack{n\in A_r\\N'_{k}\leq n< N'_{k+1}}}  v_{n+r}^p a_{m,n+r}\leq \sum_{\substack{n\in A_r\\n\geq N'_{k}+d'_{k}}}  \frac{a_{m,n+r}}{b_{n+r-N'_{k},n+r}}\leq \sum_{\substack{n\geq N'_{k}}+d'_{k}}  \frac{a_{m,n}}{b_{n-N'_{k},n}}.
\end{equation}
By the first condition in $(\widetilde{\gamma})$ we have for all $n\geq N'_{k}$
\[
 \frac{a_{m,n+1}}{b_{n+1-N'_{k},n+1}} \leq  \rho_m \frac{a_{m,n}}{b_{n-N'_{k},n}}
\]
provided that $k$ is so big that $N'_{k}\geq \nu_m$. If, in addition, $k\geq m$, then summing a geometric series gives us with \eqref{eq90} that
\begin{equation}\label{eq17}
\sum_{\substack{n\geq N'_{k}}+d'_{k}}  \frac{a_{m,n}}{b_{n-N'_{k},n}} \leq \frac{1}{1-\rho_m}\frac{a_{m,N'_{k}+d'_{k}}}{b_{d'_{k},N'_{k}+d'_{k}}}\leq \frac{1}{1-\rho_m}\frac{1}{2^k}.
\end{equation}
                                                                                                                   
Altogether it follows from \eqref{eq16} and \eqref{eq17} that, for any $r\geq 1$ and $m\geq 1$,
\[
\sum_{n\in A_r}  v_{n+r}^p a_{m,n+r}= \sum_{k\geq 0}\sum_{\substack{n\in A_r\\N'_{k}\leq n< N'_{k+1}}}  v_{n+r}^p a_{m,n+r}<\infty.
\]
This shows condition (i).

Finally, for any $m\geq 1$, let $\delta_m$ be an integer such that
\begin{equation}\label{eq17b}
\delta_m > N'_{k+1}-N'_{k},\quad 0\leq k <m.
\end{equation}
We may assume that $(\delta_m)_m$ is strictly increasing. We will then consider the sets
\[
A_r'= A_{\max(\nu_r,j_r,\delta_r)},\quad r\geq 1.
\]
It suffices now to show that condition (ii) above holds for the sequence $(A_r')_{r\geq 1}$. 

To this end, let $r,s\geq 1$. By (b') and the definition of the $A_r'$ we have that if $n\in A_r'$, $m\in A_s'$ and $N'_{k}\leq n-m < N'_{k+1}$ ($k\geq 0$) then
\begin{equation}\label{eq19a}
n-m\geq \max(\nu_r,\nu_s)=\nu_{\max(r,s)}
\end{equation}
and
\begin{equation}\label{eq19b}
n-m\geq N'_{k}+\max(j_r,j_s)=N'_{k}+j_{\max(r,s)};
\end{equation}
moreover, we can remark that $k\ge\max(r,s)$ since 
\begin{equation*}\label{eq19c}
n-m\geq N'_{k}+\max(\delta_r,\delta_s)=N'_{k}+\delta_{\max(r,s)}
\end{equation*}
and since it follows from \eqref{eq17b} that $N'_{k}+\delta_{\max(r,s)} > N'_{k+1}$ for every $k<\max(r,s)$.

Now let $m\in A'_s$ and $j=0,\ldots,r$. Then we have for $k\geq 0$, using (b') once more and also \eqref{eq19b}, that
\begin{equation}\label{eq18}
\begin{split}
\sum_{\substack{n\in A'_r\\N'_{k}\leq n-m < N'_{k+1} }} v_{n-m+j}^p a_{s,{n-m+j}} &\leq \sum_{\substack{n\in A'_r\\N'_{k}\leq n-m < N'_{k+1}}} 
\frac{a_{s,{n-m+j}}}{b_{n-m+j-N'_{k},{n-m+j}}}\\
&\leq \sum_{n-m \geq N'_{k}+ j_{\max(r,s)}} 
\frac{a_{{\max(r,s)},{n-m+j}}}{b_{n-m+j-N'_{k},{n-m+j}}},\\
&\leq \sum_{n-m \geq N'_{k}+ j_{\max(r,s)}} 
\frac{a_{{\max(r,s)},{n-m}}}{b_{n-m-N'_{k},{n-m}}}.
\end{split}
\end{equation}

Moreover, by \eqref{eq19a} and the first condition in $(\widetilde{\gamma})$, we have 
\[
\frac{a_{{\max(r,s)},{n+1-m}}}{b_{n+1-m-N'_{k},{n+1-m}}}\leq \rho_{\max(r,s)}\frac{a_{{\max(r,s)},{n-m}}}{b_{n-m-N'_{k},{n-m}}}
\]
whenever $n-m\geq N'_k$.
Thus, with \eqref{eq18} and \eqref{eq91}, we obtain after summing a geometric series that
\[
\sum_{\substack{n\in A'_r\\N'_{k}\leq n-m < N'_{k+1} }} v_{n-m+j}^p a_{s,{n-m+j}} \leq 
\frac{1}{1-\rho_{\max(r,s)}}\frac{a_{{\max(r,s)},N'_{k}+j_{\max(r,s)}}}{b_{j_{\max(r,s)},N'_{k}+j_{\max(r,s)}}}\leq\frac{1}{2^k}
\]
whenever $k\geq \max(r,s)$.

Altogether we have that
\[
\sum_{\substack{n\in A'_r\\n>m}} v_{n-m+j}^p a_{s,{n-m+j}} = \sum_{k\geq \max(r,s)} \sum_{\substack{n\in A'_r\\N'_{k}\leq n-m < N'_{k+1}} } v_{n-m+j}^p a_{s,{n-m+j}} \leq \frac{2}{2^{\max(r,s)}} = \min\Big(\frac{2}{2^r}, \frac{2}{2^s}\Big).
\]
This shows condition (ii), which completes the proof for $\lambda^p(A)$ also under condition (B) with $(\alpha)$, $(\beta)$ and $(\widetilde{\gamma})$.
\end{proof}

An example of a non-Banach space to which Theorem~\ref{thm-koefhcnotch} is applicable is the space $H(\mathbb{C})$ of entire functions, which can be identified with the Köthe sequence spaces $c_0(A)$ or $\lambda^p(A)$ for any $1\le p<\infty$ by letting $a_{m,n}=m^n$.

\begin{proposition}\label{propHC}
There exists a frequently hypercyclic weighted shift on $H(\mathbb{C})$ that is not chaotic.
\end{proposition}
\begin{proof}
Let $A$ be the Köthe matrix given by $a_{m,n}=m^n$. We consider the upper triangular matrix $B=(b_{m,n})_{m\ge 0, n\ge m}$ given by $b_{m,n}=2^{mn}$ and we show that condition (B) holds. We first remark that $B$ has increasing columns. Condition $(\alpha)$ is satisfied by considering $\mu=4m$ and $C=1$ since for every $n\ge j\ge 0$, 
\[
\frac{a_{m,n}}{a_{\mu,n+1}}=\frac{1}{4^{n+1}m}\le \frac{1}{2^{n+j+1}}= \frac{b_{j,n}}{b_{j+1,n+1}}.
\]
Condition $(\beta)$ is trivially satisfied for any $m$ since $b_{0,n}=1$ for every $n\ge 0$; finally, condition $(\gamma)$ is satisfied for $j=m$ since, for every $n\ge j$, $\frac{a_{m,n}}{b_{j,n}}=\left(\frac{m}{2^{m}}\right)^n$. It follows from Theorem~\ref{thm-koefhcnotch} that there exists a frequently hypercyclic weighted shift on $H(\mathbb{C})$ that is not chaotic. 
\end{proof}

In view of Propositions~\ref{propHD} and~\ref{propHC}, we see that among Köthe sequence spaces satisfying (N), there are examples for which frequent hypercyclicity implies chaos and others for which this is not the case. Actually, this is also the case among Köthe sequence spaces not satisfying (N) (see Example \ref{ex-nonN}).

Thus we have now a sufficient and a necessary condition for frequent hypercyclicity to imply chaos concerning weighted shifts on Köthe sequence spaces. We are still quite far from a characterization. For instance, one may ask the following.

\begin{question}\label{q-koethe}
(a) Is there a Köthe matrix $A$ so that, on $c_0(A)$, every frequently hypercyclic weighted shift is chaotic (and that there are such shifts) while $\lambda^p(A)$ admits a frequently hypercyclic, non-chaotic weighted shift? This would inverse the known behaviour for $c_0$ and $\ell^p$.

(b) Is there a Köthe matrix $A$ so that, for some $p\ge 1$, every frequently hypercyclic weighted shift is chaotic on $\lambda^p(A)$ (and that there are such shifts) while, for some $q\ne p$, $\lambda^q(A)$ admits a frequently hypercyclic, non-chaotic weighted shift? Note that, so far, all our conditions are blind to the order $p$ once it is not zero.
\end{question}

In summary, there exist a Köthe sequence space that does not support any chaotic or frequently hypercyclic weighted shift (Example \ref{ex-nohc}), one where every frequently hypercyclic weighted shift is chaotic and where there are such shifts ($\ell^p$ or $H(\mathbb{D})$ (see \cite{BaRu15} and Proposition~\ref{propHD})), and one that supports a chaotic weighted shift and a frequently hypercyclic, non-chaotic shift ($c_0$ or $H(\mathbb{C})$ (see \cite{BaGr07} and Proposition~\ref{propHC})). The only remaining case is the following.

\begin{question}
Is there a Köthe sequence space supporting frequently hypercyclic weighted shifts but no chaotic weighted shift?
\end{question}

\section{Chaotic and frequently hypercyclic weighted shifts on power series spaces}\label{sec-power} 

In this section,  we will apply the results of the previous section to so-called power series spaces which generalize the spaces $H(\D)$ and $H(\C)$, see \cite[Chapter 29]{MeVo97}. Let $\alpha= (\alpha_n)_{n\geq 0}$ be an increasing sequence of strictly positive numbers that tends to infinity, and $r\in\R\cup\{\infty\}$. 

Let $1\leq p <\infty$. Then the \textit{power series space of order $p$ and type $r$} is defined as
\[
\Lambda_{p,r}(\alpha) = \Big\{ x=(x_n)_{n\geq 0} : \text{for all } t < r, \sum_{n\geq 0} |x_n|^pe^{t\alpha_n}<\infty\Big\},
\]
while the \textit{power series space of order $0$ and type $r$} is given by
\[
C_{0,r}(\alpha) = \Big\{ x=(x_n)_{n\geq 0} : \text{for all } t < r, \lim_{n\to \infty} |x_n|e^{t\alpha_n}=0\Big\},
\]
which are topologized by the obvious (semi-)norms. If $r=\infty$ then the space is said to be of \textit{infinite type}, otherwise of \textit{finite type}. Power series spaces are particular K\"othe sequence spaces.

Many interesting spaces are or can be considered as power series spaces, for example the space $s$ of rapidly decreasing sequences (any $p$, $r=\infty$, $\alpha=(\log(n+2))_n$), the space $H(\C)$ of entire functions (any $p$, $r=\infty$, $\alpha=(n+1)_n$) and the space $H(\D_R)$ of holomorphic functions on the disk of radius $R>0$ (any $p$, $r=\log R$, $\alpha=(n+1)_n$). A recent addition is the space $\text{ces}(p+)$, $1\leq p<\infty$, which was studied in \cite{ABR18} and shown to coincide with $\Lambda_{1,\frac{1}{p}-1}(\alpha)$ for $\alpha=(\log(n+2))_n$.

In the literature, authors are often content with considering one particular order, like $p=2$ (see \cite[Chapter 29]{MeVo97}) or $p=1$; but see, for example, \cite{DuRa71} for the full family.

\subsection{Power series spaces of infinite type}\label{subsec-pss_infinite}

Power series spaces of infinite type are special K\"othe sequence spaces with K\"othe matrix $A=(a_{m,n})_{m,n}$ given by
\begin{equation*}                                                                                              
a_{m,n}=m^{\alpha_n}, \ m\geq 1, n\geq 0.
\end{equation*}

The following is easily verified, see Remark \ref{rem-KoBa}, Propositions \ref{caracc_0} and \ref{caracNuc}, and the proof of \cite[Proposition 29.6]{MeVo97}.

\begin{proposition}\label{prop-proppssinf}
Let $X=\Lambda_{p,\infty}(\alpha)$, $1\leq p<\infty$, or $X=C_{0,\infty}(\alpha)$ be a power series space of infinite type. Then it is a non-Banach Fr\'echet space for which the basis $(e_n)_n$ is boundedly complete. In addition, its K\"othe matrix satisfies condition \emph{(N)} if and only if
\[
\sup_{n\geq 1} \frac{\log(n)}{\alpha_n}<\infty.
\]
\end{proposition}

Moreover we know from Subsection \ref{sec-exist} that power series spaces always support weighted shifts. As for the existence of (frequently) hypercyclic or chaotic weighted shifts we have the following.

\begin{proposition}\label{prop-pssinf}
Let $X=\Lambda_{p,\infty}(\alpha)$, $1\leq p<\infty$, or $X=C_{0,\infty}(\alpha)$ be a power series space of infinite type. 

\emph{(a)} Then $X$ supports a hypercyclic weighted shift if and only if 
\[
\limsup_{N\to\infty} \frac{\sum_{n=0}^{N-1}\alpha_n}{\alpha_N}=\infty.
\] 

\emph{(b)} Moreover, $X$ supports a chaotic weighted shift if and only if
\[
\lim_{N\to\infty} \frac{\sum_{n=0}^{N-1}\alpha_n}{\alpha_N}=\infty.
\]
\end{proposition}

\begin{proof}
It is readily verified that the K\"othe matrix $(m^{\alpha_n})_{m,n}$ satisfies condition \eqref{eq10}.

(a) Suppose that $\limsup_{N\to\infty} \frac{\sum_{n=0}^{N-1}\alpha_n}{\alpha_N}=\infty$. Choose $\mu\geq 2$ and $C\geq 1$. Writing any given $m\geq 1$ as $m=\mu^\rho$ with some $\rho\geq 0$ we see that 
\[
\liminf_{N\to\infty} \frac{m^{\alpha_N}}{C^N \mu^{\sum_{n=1}^{N}\alpha_n}}=\liminf_{N\to\infty} \frac{1}{C^N \mu^{\alpha_N\big(\frac{1}{\alpha_N}\sum_{n=1}^{N-1}\alpha_n+1-\rho\big)}}=0.
\]
Hence, by Corollary \ref{cor-exhcko}, there exists a hypercyclic weighted shift on $X$.

On the other hand, if $\limsup_{N\to\infty} \frac{\sum_{n=0}^{N-1}\alpha_n}{\alpha_N}<\infty$ then, since $\alpha$ is an increasing sequence, we also have $\limsup_{N\to\infty} \frac{N}{\alpha_N}\le \limsup_{N\to\infty} \frac{1}{\alpha_0}\frac{\sum_{n=0}^{N-1}\alpha_n}{\alpha_N}<\infty$. Let $\mu\ge 1$ and $C>0$, and choose 
\[
m\geq \sup_{N\geq 0} C^{\frac{N}{\alpha_N}} \mu^{\frac{1}{\alpha_N}\sum_{n=1}^{N}\alpha_n}.
\]
Then we have that, for all $N\geq 0$,
\[
\frac{m^{\alpha_N}}{C^N \mu^{\sum_{n=1}^{N}\alpha_n}}=\Big(\frac{m}{C^{\frac{N}{\alpha_N}} \mu^{\frac{1}{\alpha_N}\sum_{n=1}^{N}\alpha_n}}\Big)^{\alpha_N}\geq 1.
\]
Hence, by Corollary \ref{cor-exhcko}, $X$ admits no hypercyclic weighted shift.

(b) The proof, based this time on Corollary \ref{cor-exchako}, is very similar and therefore omitted. 
\end{proof}

In the sequel we will be led to study properties of power series spaces based on the behaviour of the quotients
\[
\frac{\alpha_{n+1}}{\alpha_n} 
\]
as $n\to\infty$; note that these quotients are all at least 1. Thus the following is of interest.

\begin{example}\label{ex infinite type}
(a) Suppose that
\[
\liminf_{n\to\infty} \frac{\alpha_{n+1}}{\alpha_n}>1.
\]
Then the corresponding power series spaces of infinite type do not support any hypercyclic weighted shifts. Indeed, there is then some $n_0\geq 0$ and $\rho>1$ such that $\frac{\alpha_{n+k}}{\alpha_n}\geq \rho^k$ for all $n\geq n_0$ and $k\geq 0$. Hence, for any $N\geq n_0$,
\[
\frac{\sum_{n=n_0}^{N-1}\alpha_n}{\alpha_N}\le \sum_{k=1}^{N-n_0}\frac{1}{\rho^k}\leq \frac{1}{\rho-1},
\]
so that the assertion follows from the previous proposition.

(b) On the other hand, if
\[
\lim_{n\to\infty} \frac{\alpha_{n+1}}{\alpha_n}=1
\]
then any corresponding power series spaces of infinite type supports a chaotic weighted shift. In fact, in that case, for every $\varepsilon>0$, there exists some $n_0\geq 0$ such that $\frac{\alpha_n}{\alpha_{n+k}}\ge \frac{1}{(1+\varepsilon)^{k}}$ for every $n\ge n_0$ and $k\ge 0$, and thus we have that
\[
\liminf_{N\to\infty}\frac{\sum_{n=n_0}^{N-1}\alpha_n}{\alpha_N}\ge  \liminf_{N\to\infty}\sum_{k=1}^{N-n_0}\frac{1}{(1+\varepsilon)^{k}}=\frac{1}{\varepsilon}.
\]
Since this inequality holds for every $\varepsilon>0$, the assertion follows again from the previous proposition.
\end{example}

We return to the question whether every frequently hypercyclic weighted shift is chaotic. In the context of power series spaces of infinite type, condition (TK) turns out to be too strong and does not allow the existence of hypercyclic weighted shift, as shown by the following result.

\begin{theorem}\label{thm-pssinf}
Let $X=\Lambda_{p,\infty}(\alpha)$, $1\leq p<\infty$, or $X=C_{0,\infty}(\alpha)$ be a power series space of infinite type. If condition \emph{(TK)} holds then $X$ does not support hypercyclic weighted shifts.
\end{theorem}
\begin{proof}
Let $v_n=\frac{1}{2^{\alpha_1+\ldots+\alpha_n}}$ for $n\geq 0$. Then, for any $m\geq 1$, when choosing $\mu= 2m$, we have for any $n\geq 1$ that
\[
v_{n-1} a_{m,n-1} = \frac{1}{2^{\alpha_1+\ldots+\alpha_{n-1}}} m^{\alpha_{n-1}}= v_n 2^{\alpha_n}m^{\alpha_{n-1}}\leq v_n(2m)^{\alpha_n} =v_n a_{\mu,n}. 
\]
Thus, by (TK), taking $m = 1$, there is some $\mu\geq 1$ such that for any $j\geq 1$ there is some constant $C_j>0$ such that, for any $n\geq j$,
\[                                                                                                              
\frac{1}{2^{\alpha_{1}+\ldots+\alpha_{n-j}}}\leq C_j\frac{1}{2^{\alpha_1+\ldots+\alpha_{n}}} \mu^{\alpha_n},
\]
hence
\[                                                                                                              
2^{\alpha_{n-j+1}+\ldots+\alpha_{n}}\leq C_j \mu^{\alpha_n},
\]
and thus
\[
j\alpha_{n-j}\le \alpha_{n-j}+\ldots+\alpha_{n-1}\leq\alpha_{n-j+1}+\ldots+\alpha_{n}\leq \log_2 C_j + \alpha_n \log_2\mu.
\]
We then have that, for any $j\geq 1$,
\[
\limsup_{n\to\infty} \frac{\alpha_{n-j}}{\alpha_n}\le \limsup_{n\to\infty} \Big(\frac{\log_2 C_j}{j\alpha_n}+\frac{\log_2\mu}{j}\Big)=\frac{\log_2\mu}{j}.
\]
We choose $j\geq 1$ such that $\frac{\log_2\mu}{j} < \frac{1}{2}$. Then there is some $n_0\geq j$ such that, for every $n\geq n_0$,
\[
\frac{\alpha_{n-j}}{\alpha_n}\le \frac{1}{2}.
\]
It follows that 
\begin{align*}
\limsup_{N\to \infty} \frac{\sum_{n=0}^{N-1} \alpha_n}{\alpha_N}&\le \limsup_{N\to \infty} \Bigg(\frac{1}{\alpha_N} \sum_{\nu=0}^{\lfloor \frac{N-n_0}{j} \rfloor} \sum_{n=N-(\nu+1)j+1}^{N-\nu j} \alpha_n\Bigg)\\
&\le \limsup_{N\to \infty}  \Bigg(\sum_{\nu=0}^{\lfloor \frac{N-n_0}{j} \rfloor} j \frac{\alpha_{N-\nu j}}{\alpha_N}\Bigg)\le \limsup_{N\to \infty} \Bigg(j\sum_{\nu=0}^{\lfloor \frac{N-n_0}{j} \rfloor } \frac{1}{2^\nu}\Bigg)<\infty.
\end{align*}
We then deduce from Proposition~\ref{prop-pssinf} that $X$ does not support hypercyclic weighted shifts.
\end{proof}

This result is surely disappointing. It shows that our sufficient condition for frequent hypercyclicity to imply chaos obtained in Subsection \ref{subsec-Frechet} is not good enough in the present context; recall that, for Köthe sequence spaces, (TK) is equivalent to (T). 

\begin{question} 
Does there exist a power series space of infinite type that supports chaotic weighted shifts and for which every frequently hypercyclic weighted shift is chaotic?
\end{question}

In the opposite direction we have the following result.

\begin{theorem}\label{thm-pssinfbis}
Let $X=\Lambda_{p,\infty}(\alpha)$, $1\leq p<\infty$, or $X=C_{0,\infty}(\alpha)$ be a power series space of infinite type. Suppose that either
\[
\begin{cases} \displaystyle\exists \varepsilon\in(0,1),\, \forall C>0,\, \exists j\geq 1:\sum_{n\geq j} \varepsilon^{\alpha_{n-j}+\ldots+\alpha_{n-1} - C\alpha_n}<\infty, &\text{if $X=\Lambda_{p,\infty}(\alpha)$},\\
\forall C>0,\, \exists j\geq 1:\displaystyle\liminf_{n\to\infty} (\alpha_{n-j}+\ldots+\alpha_{n-1} - C\alpha_n)>-\infty, &\text{if $X=C_{0,\infty}(\alpha)$};
\end{cases}
\]                                                               
or
\[
\lim_{n\to\infty}\frac{\alpha_{n+1}}{\alpha_n}=1.
\]
Then $X$ supports a chaotic weighted shift, and there exists a frequently hypercyclic weighted shift on $X$ that is not chaotic.
\end{theorem}

\begin{proof} We will apply Theorem \ref{thm-koefhcnotch}. 
First of all, for an arbitrary $\varepsilon\in (0,1)$, define the upper triangular matrix $B=(b_{m,n})$ by
\[
b_{m,n} = \frac{1}{\varepsilon^{\alpha_{n-m}+\ldots+\alpha_{n-1}}}\quad n\geq m\geq 0,
\]
where we interpret $b_{0,n}=1$, $n\geq 0$. Then $B$ has increasing columns. Let $m\geq 1$, and choose $\mu\geq m/\varepsilon$. Then we have for $n\geq j\geq 0$,
\[
\frac{a_{m,n}}{a_{\mu,n+1}}\frac{b_{j+1,n+1}}{b_{j,n}} = \frac{(m/\varepsilon)^{\alpha_n}}{\mu^{\alpha_{n+1}}}  \leq 1,
\]
so that part ($\alpha$) of condition (B) holds. Part ($\beta$) is trivially satisfied. 

Now suppose that the first hypothesis holds, and let $X=\Lambda_{p,\infty}(\alpha)$. Then choose as $\varepsilon$ a corresponding value. Let $m\geq1$, and let $C>0$ be such that $m=\varepsilon^{-C}$. Then 
\[
\frac{a_{m,n}}{b_{j,n}} =\varepsilon^{\alpha_{n-j}+\ldots+\alpha_{n-1}-C\alpha_n},\quad n\geq j\geq 0.
\]
Hence the hypothesis implies that also part ($\gamma$) of condition (B) holds in this case. For $X=C_{0,\infty}(\alpha)$, one first needs to note that the hypothesis also holds for $2C$, so that we have, for the $j\geq 1$ given by the hypothesis, 
\[
\alpha_{n-j}+\ldots+\alpha_{n-1}-C\alpha_n = (\alpha_{n-j}+\ldots+\alpha_{n-1}-2C\alpha_n) + C\alpha_n \to \infty,
\]
which implies part ($\gamma$) also in this case by the same argument as above.

Finally suppose that the second hypothesis is satisfied, where now the space $X$ is arbitrary. For simplicity let us take $\varepsilon=\frac{1}{2}$. Let $m\geq 1$, and choose $\eta>0$ such that $m^{1+\eta}<2m$. Then there is some $N\geq 0$ such that $\alpha_{n+1}\leq (1+\eta)\alpha_n$ for all $n\geq N$. Hence we have for all $j\geq 0$ and $n\geq \max(N,j)$ that
\[
\frac{a_{m,n+1}}{a_{m,n}}\frac{b_{j,n}}{b_{j+1,n+1}} = \frac{m^{\alpha_{n+1}}}{(2m)^{\alpha_{n}}}  \leq \Big(\frac{m^{1+\eta}}{2m}\Big)^{\alpha_{n}}\leq \Big(\frac{m^{1+\eta}}{2m}\Big)^{\alpha_{0}}<1,
\]
which shows the first condition in $(\widetilde{\gamma})$. 

Next we fix a strictly increasing sequence $(d_k)_{k\geq 1}$ of positive integers such that, for all $k\geq 1$, $2^{d_k}\geq k^2(1+k)$. We then choose $\eta_k>0$ such that $(1+\eta_k)^{d_k}\leq 2$ for $k\geq 1$. Finally, there are $\nu_k\geq 0$ such that $\alpha_{n+1}\leq (1+\eta_k)\alpha_n$ for all $k\geq 1$ and $n\geq \nu_k$, and we can assume that $(\nu_k)_k$ is strictly increasing. Thus we have for $k\geq 1$, $n\geq \nu_k$ and $1\leq m\leq k$ that
\[
\frac{a_{m,n+d_k}}{b_{d_k,n+d_k}}= \frac{m^{\alpha_{n+d_k}}}{2^{\alpha_n+\ldots+\alpha_{n+d_k-1}}}\leq \frac{m^{(1+\eta_k)^{d_k}\alpha_n}}{2^{d_k\alpha_n}}\leq \Big(\frac{m^2}{2^{d_k}}\Big)^{\alpha_n}\leq \Big(\frac{1}{k+1}\Big)^{\alpha_n}\leq \Big(\frac{1}{k+1}\Big)^{\alpha_0}.
\]
This shows the second condition in $(\widetilde{\gamma})$ with $n_k=\nu_k+d_k$, $k\geq 1$.

Taking $j_m=d_m$ for $m\geq 1$, and choosing $m=k$ above we have that, for all $m\geq 1$ and $n\geq n_m$,
\[
\frac{a_{m,n+j_m}}{b_{j_m,n+j_m}}\leq \Big(\frac{1}{m+1}\Big)^{\alpha_n}.
\]
Thus also the third condition in $(\widetilde{\gamma})$ holds.
\end{proof} 

Thus we have by Example \ref{ex infinite type} and Theorem \ref{thm-pssinfbis} the following.

\begin{corollary}\label{corpssinf}
Let $X=\Lambda_{p,\infty}(\alpha)$, $1\leq p<\infty$, or $X=C_{0,\infty}(\alpha)$ be a power series space of infinite type. 

{\rm (a)} If
\[
\liminf_{n\to\infty} \frac{\alpha_{n+1}}{\alpha_n}>1
\]
then $X$ supports no hypercyclic weighted shift.

{\rm (b)} If 
\[
\lim_{n\to\infty} \frac{\alpha_{n+1}}{\alpha_n}=1
\]
then $X$ supports a chaotic weighted shift, and there exists a frequently hypercyclic weighted shift on $X$ that is not chaotic.
\end{corollary}

Corollary \ref{corpssinf}(b) applies in particular to the space $H(\mathbb{C})$ of entire functions, see also Proposition \ref{propHC}, and the space $s$.

\begin{proposition}
The space $s$ of rapidly decreasing sequences admits a frequently hypercyclic weighted shift that is not chaotic.
\end{proposition}

We end our study of power series spaces of infinite type by giving an example of such a space supporting chaotic weighted shift but to which Theorem \ref{thm-pssinfbis} is not applicable.

\begin{example}\label{ex-open} 
Consider the sequence space
\[
X=\Big\{ (x_n)_n : \forall m\geq 1,\, \Big(\sup_{2^{k-1}<n\leq 2^k}|x_n|\Big) m^{k!} \to 0 \text{ as $k\to\infty$}\Big\}.
\]
This is a power series space $C_{0,\infty}(\alpha)$ with $\alpha$ given by
\[
\alpha_n = k!\quad\text{if $2^{k-1}<n\leq 2^k$}, k\geq 1,
\]
and $\alpha_0=1$. Then we have for $2^{k-1}<N\leq 2^k$, $k\geq 2$, that
\[
\frac{\sum_{n=0}^{N-1}\alpha_n}{\alpha_N}\geq \frac{\sum_{n=2^{k-2}+1}^{2^{k-1}}\alpha_n}{k!}= \frac{2^{k-2}(k-1)!}{k!}\to\infty
\]
as $N\to \infty$, so that by Proposition \ref{prop-pssinf} the space $C_{0,\infty}(\alpha)$ admits chaotic and hence frequently hypercyclic weighted shifts.

On the other hand, we have for any $j\geq 1$ and any $k$ sufficiently large that
\[
\alpha_{2^k+1-j}+\ldots+\alpha_{2^k}-\alpha_{2^k+1} = jk! - (k+1)!,
\]
so that $\liminf_{n\to\infty} (\alpha_{n-j}+\ldots+\alpha_{n-1} - \alpha_n)=-\infty$. And, of course, $\limsup_{n\to\infty} \frac{\alpha_{n+1}}{\alpha_n}>1$.
Theorem \ref{thm-pssinfbis} is thus not applicable, and we do not know if every frequently hypercyclic weighted shift on $C_{0,\infty}(\alpha)$ is chaotic or not.
\end{example}

\subsection{Power series spaces of finite type}\label{subsec-pss_finite}

We turn to power series spaces of finite type. We first note that, since $e^{t\alpha_n}= e^{r\alpha_n} e^{(t-r)\alpha_n}$, one can change the finite type $r$ by a diagonal transform into any other finite type; hence any weighted shift on a power series spaces of finite type is conjugate to a weighted shift on a power series spaces of type $0$. In the sequel it will therefore suffice to do the proofs in the case of $r=0$.

Now, power series spaces of type $0$ are special K\"othe sequence spaces with K\"othe matrix $A=({m,n})_{m,n}$ given by
\begin{equation*}                                                                                              
a_{m,n}=\frac{1}{\big(1+\frac{1}{m}\big)^{\alpha_n}}, \ m\geq 1, n\geq 0.
\end{equation*}

We again have the following, see Remark \ref{rem-KoBa}, Propositions \ref{caracc_0} and \ref{caracNuc}, and the proof of \cite[Proposition 29.6]{MeVo97}.

\begin{proposition}\label{prop-proppssfin}
Let $X=\Lambda_{p,r}(\alpha)$, $1\leq p<\infty$, or $X=C_{0,r}(\alpha)$ be a power series space of finite type. Then it is a non-Banach Fr\'echet space for which the basis $(e_n)_n$ is boundedly complete. In addition, its K\"othe matrix satisfies condition \emph{(N)} if and only if
\[
\lim_{n\to\infty} \frac{\log(n)}{\alpha_n}=0.
\]
\end{proposition}

Moreover, by Subsection \ref{sec-exist}, these spaces always support weighted shifts. The characterizing conditions of Propositions \ref{prop-exhcko} and \ref{prop-exchako} for the existence of hypercyclic or chaotic weighted shifts do not seem to simplify easily. Thus we will be content with the following.

\begin{example}\label{ex-pssfin}
(a) Suppose that
\[
\lim_{n\to\infty}\frac{\alpha_{n+1}}{\alpha_n}=\infty.
\]
Then the corresponding power series spaces of  finite type do not support any hypercyclic weighted shifts. We will show this for $\Lambda_{p,r}(\alpha)$, the argument for $C_{0,r}(\alpha)$ being similar. 

Thus, let $B_w$ be a weighted shift on $\Lambda_{p,r}(\alpha)$. By continuity there then exist $\mu\ge 1$ and $C>1$ such that, for every $n\ge 1$, $|w_n|^p a_{1,n-1}\le C a_{\mu,n}$ and thus $|w_n|^p\le \frac{C2^{\alpha_{n-1}}}{(1+\frac{1}{\mu})^{\alpha_n}}$. Therefore we get for $n\geq 1$ that
\[
\frac{a_{2\mu,n}}{\prod_{k=1}^n|w_k|^p}\ge \frac{1}{C^{n}}\frac{1}{(1+\frac{1}{2\mu})^{\alpha_n}}\prod_{k=1}^{n} \frac{(1+\frac{1}{\mu})^{\alpha_k}}{2^{\alpha_{k-1}}} =\Big(\frac{1+\frac{1}{\mu}}{C^{n/\alpha_n}2^{\alpha_{n-1}/\alpha_n}(1+\frac{1}{2\mu})}\Big)^{\alpha_n}\prod_{k=1}^{n-1} \Big(\frac{1+\frac{1}{\mu}}{2^{\alpha_{k-1}/\alpha_k}}\Big)^{\alpha_k}.
\]
Now, since $\frac{\alpha_{n+1}}{\alpha_n}\to\infty$ and therefore $\frac{n}{\alpha_n}\to 0$ as $n\to\infty$ there is some $N\geq 0$ so that, for all $n\geq N$,
\[
\frac{1+\frac{1}{\mu}}{C^{n/\alpha_n}2^{\alpha_{n-1}/\alpha_n}(1+\frac{1}{2\mu})}\geq 1\quad\text{and}\quad\frac{1+\frac{1}{\mu}}{2^{\alpha_{n-1}/\alpha_n}}\geq 1.
\]
Thus we have for all $n\geq N$,
\[
\frac{a_{2\mu,n}}{\prod_{k=1}^n|w_k|^p}\ge\prod_{k=1}^{N-1} \Big(\frac{1+\frac{1}{\mu}}{2^{\alpha_{k-1}/\alpha_k}}\Big)^{\alpha_k}>0,
\]
which contradicts \cite[Theorem 4.8]{GrPe11}. We conclude that $B_w$ is not hypercyclic.

(b) On the other hand, if
\[
\limsup_{n\to\infty }\frac{\alpha_{n+1}}{\alpha_n}<\infty
\] 
then any corresponding power series space of finite type supports a chaotic weighted shift. Specifically, we will show that the weighted shift $B_w$ with $w_n=2$ for all $n\geq 1$ is chaotic on any such space, which we will prove again only for $\Lambda_{p,r}(\alpha)$. To this end, let $m\geq 1$. It then follows from the hypothesis that there are $n_m\geq 0$ and $\mu\ge 1$ such that, for every $n\ge n_m$,
\[
\frac{2^p}{(1+\frac{1}{m})^{\alpha_{n-1}}}=\Big(\frac{2^{p/\alpha_n}}{(1+\frac{1}{m})^{\alpha_{n-1}/\alpha_n}}\Big)^{\alpha_n}\le \frac{1}{(1+\frac{1}{\mu})^{\alpha_n}};
\]
note that $\alpha_n\to\infty$ as $n\to\infty$. Thus there is a constant $C>0$ such that, for all $n\geq 1$, 
\[
w_n^p a_{m,n-1}\le C a_{\mu,n},
\]
so that $B_w$ is an operator on $\Lambda_{p,r}(\alpha)$. Moreover, for every $m\ge 1$, we have that
\[
\sum_{n\geq 0}\frac{1}{(w_1\cdots w_n)^p}a_{m,n}\le \sum_{n\geq 0}\frac{1}{2^{np}}<\infty,
\]
which shows that $B_w$ is chaotic; see \eqref{eq0}.
\end{example}

Concerning the condition (TK), power series spaces of finite type behave quite differently from their counterparts of infinite type, see 
Theorem \ref{thm-pssinf}.

\begin{theorem}\label{thm-pssfin}
Let $X=\Lambda_{p,r}(\alpha)$, $1\leq p<\infty$, or $X=C_{0,r}(\alpha)$ be a power series space of finite type. Suppose that
\[
\limsup_{j\to\infty}\limsup_{n\to\infty} \frac{\alpha_{n+j}}{\alpha_n} <\infty.
\]
Then $X$ supports a chaotic weighted shift.

Moreover, for any weighted shift $B_w$ on $X$, the following assertions are equivalent:
\begin{enumerate}
\item[\rm (i)] $B_{w}$ is $\UU$-frequently hypercyclic on $X$;
\item[\rm (ii)] $B_{w}$ is frequently hypercyclic on $X$;
\item[\rm (iii)] $B_{w}$ is chaotic on $X$;
\item[\rm (iv)] the series $\sum _{n\geq 0}\frac{1}{w_1\cdots w_n}e_n$ is convergent in $X$.
\end{enumerate}
\end{theorem}

\begin{proof} Since $\alpha$ is increasing, the first limsup in the hypothesis is a supremum. Thus, the first assertion follows from Example \ref{ex-pssfin}(b). In addition, there is some $\delta>0$ such that, for all $j\geq 1$, we have that
\begin{equation}\label{eq-del}
\delta{\alpha_n}\leq \alpha_{n-j}
\end{equation}
holds for all sufficiently large $n$.

To complete the proof it suffices, by Theorem \ref{thm-kothe}, to show that condition (TK) is satisfied. To this end, let $(v_n)_{n\geq 0}$ be a strictly positive sequence so that, for any $m\geq 1$, there are $\mu_m\geq 1$ and $C_m>0$ such that, for all $n\geq 1$,
\[
v_{n-1} a_{m,n-1} \leq C_m v_n a_{\mu_m,n}, 
\]
hence
\[
v_{n-1}  \leq C_m v_n \frac{(1+\frac{1}{m})^{\alpha_{n-1}}}{(1+\frac{1}{\mu_m})^{\alpha_{n}}}\leq C_m v_n \frac{(1+\frac{1}{m})^{\alpha_{n}}}{(1+\frac{1}{\mu_m})^{\alpha_{n}}}.
\]
Let $(n_m)_m$ be a strictly increasing sequence of positive integers such that, for all $n\geq n_m$,
\[
\frac{C_m}{(1+\frac{1}{\mu_m})^{\alpha_{n}}}\leq 1.
\]
Thus, if we define $M_n= \log (1+\frac{1}{m})$ for $n_m\leq n< n_{m+1}$, $m\geq 1$, with $M_n=\log 2$ for $0\leq n<n_1$, then $(M_n)_n$ is a decreasing positive sequence that converges to 0 and such that, for all $n\geq n_1$,
\[
v_{n-1} \leq v_n e^{M_n\alpha_n}.
\]

Now let $m\geq 1$. Then there is some $\mu\geq 1$ such that
\begin{equation}\label{eq-del2}
\log(1+\tfrac{1}{\mu})< \delta\log(1+\tfrac{1}{m}).
\end{equation}
Let $j\geq 1$. Then we have for all $n\geq n_1+j$ that
\begin{align*}
v_{n-j}\frac{a_{m,n-j}}{a_{\mu,n}} &\leq v_n\frac{1}{(1+\frac{1}{m})^{\alpha_{n-j}}}\exp(M_{n-j+1}\alpha_{n-j+1}+\ldots+M_{n}\alpha_{n})(1+\tfrac{1}{\mu})^{\alpha_n}\\
&\leq v_n\exp(-\log(1+\tfrac{1}{m})\alpha_{n-j}+ jM_{n-j}\alpha_{n}+\log(1+\tfrac{1}{\mu})\alpha_n).
\end{align*}
Since $(M_n)_n$ tends to $0$, we have by \eqref{eq-del2} that for all $n$ sufficiently large
\[
v_{n-j}\frac{a_{m,n-j}}{a_{\mu,n}}\le v_n\exp(-\log(1+\tfrac{1}{m})\alpha_{n-j}+ \delta\log(1+\tfrac{1}{m})\alpha_n);
\]
hence, by \eqref{eq-del} we have for all large $n$ that 
\[
v_{n-j}\frac{a_{m,n-j}}{a_{\mu,n}}\leq v_n.
\] 
This shows that (TK) holds.
\end{proof}

As for the existence of non-chaotic frequently hypercyclic weighted shifts we have the following.

\begin{theorem}\label{thm-pssfinbis}
Let $X=\Lambda_{p,r}(\alpha)$, $1\leq p<\infty$, or $X=C_{0,r}(\alpha)$ be a power series space of finite type. Suppose that
\[
\limsup_{n\to\infty} \frac{\alpha_{n+1}}{\alpha_n} < \infty
\]
and
\[
\begin{cases}
\displaystyle\exists \varepsilon \in (0,1),\, \forall \delta>0,\,\exists j\geq 1:\,\sum_{n\geq j}\varepsilon^{-\alpha_{n-j}+\delta\alpha_n} <\infty, 
&\text{if $X=\Lambda_{p,r}(\alpha)$},\\
\displaystyle\limsup_{j\to\infty} \liminf_{n\to\infty} \frac{\alpha_{n+j}}{\alpha_{n}}= \infty, &\text{if $X=C_{0,r}(\alpha)$}.
\end{cases}
\]                          
Then $X$ supports a chaotic weighted shift, and there exists a frequently hypercyclic weighted shift on $X$ that is not chaotic.
\end{theorem}

\begin{proof} 
We will show that the hypothesis implies condition (B) under the alternative ($\gamma$) in Theorem \ref{thm-koefhcnotch}. 

We first consider the case when $X=\Lambda_{p,r}(\alpha)$. Let $\varepsilon\in (0,1)$ be given by the hypothesis. We then define the upper triangular matrix $B=(b_{m,n})$ by
\[
b_{m,n} = \varepsilon^{\alpha_{n-m}},\quad n\geq m\geq 0.
\]
Note that $B$ has increasing columns. Let $m\geq 1$, and choose $\mu\geq 1$ such that
\[
(1+\tfrac{1}{\mu})^M \leq 1+\tfrac{1}{m},
\]
where $M=\sup_{n\geq 0}\frac{\alpha_{n+1}}{\alpha_n}$. Then we have for $n\geq j\geq 0$,
\[
\frac{a_{m,n}}{a_{\mu,n+1}}\frac{b_{j+1,n+1}}{b_{j,n}} = \frac{(1+\frac{1}{\mu})^{\alpha_{n+1}}}{(1+\frac{1}{m})^{\alpha_{n}}}
 \leq \Big(\frac{(1+\frac{1}{\mu})^M}{1+\frac{1}{m}}\Big)^{\alpha_n}\leq 1,
\]
so that part ($\alpha$) of condition (B) holds. 

Next, choose $m$ so that $\varepsilon(1+\frac{1}{m})\leq 1$. Then
\[
\frac{a_{m,n}}{b_{0,n}} = \frac{1}{\varepsilon^{\alpha_n}(1+\frac{1}{m})^{\alpha_n}} \geq 1
\]
for all $n\geq 0$, so that also part ($\beta$) holds.

Finally, let $m\geq 1$, and define $\delta>0$ by $\varepsilon^{\delta}=\frac{1}{1+\frac{1}{m}}$. Then we have for $n\geq j\geq 1$,
\[
\frac{a_{m,n}}{b_{j,n}} =\varepsilon^{-\alpha_{n-j}+\delta\alpha_n},
\]
so that by considering the integer $j$ given in the hypothesis for $\varepsilon$ and $\delta$, we can deduce that also ($\gamma$) holds. This completes the proof for $X=\Lambda_{p,r}(\alpha)$.

If $X=C_{0,r}(\alpha)$, it is easy to see that the hypothesis implies that
\[
\forall \delta>0,\,\exists j\geq 1:\,\lim_{n\to\infty}\varepsilon^{-\alpha_{n-j}+\delta\alpha_n} =0, 
\]
where $\varepsilon\in(0,1)$ can be chosen arbitrarily; see also the proof of the following corollary. Then we proceed as in the case of $X=\Lambda_{p,r}(\alpha)$.  
\end{proof}

The following is immediate from the previous results.

\begin{corollary}\label{corpssfin}
Let $X=\Lambda_{p,r}(\alpha)$, $1\leq p<\infty$, or $X=C_{0,r}(\alpha)$ be a power series space of finite type. 

{\rm (a)} If
\[
\lim_{n\to\infty}\frac{\alpha_{n+1}}{\alpha_{n}} = \infty
\]
then $X$ supports no hypercyclic weighted shift.

{\rm (b)} If
\[
\limsup_{j\to\infty} \limsup_{n\to\infty}\frac{\alpha_{n+j}}{\alpha_{n}} < \infty
\]
then $X$ supports a chaotic weighted shift, and every frequently hypercyclic weighted shift on $X$ is chaotic.

{\rm (c)} If 
\[
\limsup_{n\to\infty} \frac{\alpha_{n+1}}{\alpha_{n}}< \infty\quad\text{and}\quad
\limsup_{j\to\infty} \liminf_{n\to\infty} \frac{\alpha_{n+j}}{\alpha_{n}}= \infty
\]
then $X$ supports a chaotic weighted shift, and there exists a frequently hypercyclic weighted shift on $X$ that is not chaotic.
\end{corollary}

\begin{proof} In view of Example \ref{ex-pssfin} and the Theorems \ref{thm-pssfin} and \ref{thm-pssfinbis}, we need only show that, for $X=\Lambda_{p,r}(\alpha)$, the hypothesis in (c) implies the one in Theorem \ref{thm-pssfinbis}. Thus, let $\varepsilon\in(0,1)$ be arbitrary, and let $\delta>0$, where we may assume that $\delta\leq\frac{1}{2}$. Then the hypothesis in (c) implies that there are some $j\geq 1$ and some $N_j\geq j$ such that, for all $n\geq N_j$, $\alpha_{n-j}\leq \delta \alpha_n$. This implies that
\[
\sum_{n\geq N_j}\varepsilon^{-\alpha_{n-j}+2\delta\alpha_n} \leq \sum_{n\geq N_j}\varepsilon^{\delta\alpha_n}.
\]
But since we also have that $\frac{\alpha_{n+j}}{\alpha_n}\geq 2$ for all large $n$, the latter series converges, which had to be shown.
\end{proof}

Corollary \ref{corpssfin}(b) applies in particular to the space $H(\mathbb{D})$ of holomorphic functions on the unit disk $\mathbb{D}$, see also Proposition \ref{propHD}.

The following table summarizes our findings for power series spaces $\Lambda_{p,r}(\alpha)$, $1\leq p<\infty$, and $C_{0,r}(\alpha)$ of finite or infinite type when $\alpha$ satisfies
\[
\rho:=\lim_{n\to\infty}\frac{\alpha_{n+1}}{\alpha_n}\in[1,+\infty],
\]
see Corollaries \ref{corpssinf} and \ref{corpssfin}. In this table, '$=$' stands for the fact that every frequently hypercyclic weighted shift is chaotic (and that there are chaotic and hence frequently hypercyclic weighted shifts), '$\neq$' for the fact that some frequently hypercyclic weighted shifts are not chaotic, and '$\times$' indicates that there are no hypercyclic (and hence no frequently hypercyclic) weighted shifts.\\

\begin{center}
\begin{tabular}{c||c|c||c|c}
 & $r<\infty$ & example & $r=\infty$ & example\\
\hline\hline
$\rho=1$         & =        & $H(\D)$ & $\neq$   & $H(\C)$, $s$\\
$1<\rho<\infty $ & $\neq$   &         & $\times$ & \\
$\rho=\infty$    & $\times$ &         & $\times$ & 
\end{tabular}
\end{center}
~\\

The specific examples in the table are all nuclear spaces. We end by adding some non-nuclear examples. This will support the claim made at the end of the Introduction.

\begin{example}\label{ex-nonN}
(a) Let the K\"othe matrix $A$ be given by
\[
a_{m,n}=\frac{1}{(n+1)^{1/m}},\quad m\geq 1, n\geq 0. 
\]
Then the corresponding K\"othe sequence spaces are power series spaces of type $0$ with $\alpha_n=\log(n+1)$, $n\geq 1$. Thus, on any such space, there are frequently hypercyclic weighted shifts, and they are all chaotic since $\lim_n \frac{\alpha_{n+1}}{\alpha_n}= 1$. Since condition (N) is not satisfied, see Proposition \ref{prop-proppssfin}, all the spaces $c_0(A)$ and $\lambda^p(A)$, $1\leq p <\infty$, are different by Proposition~\ref{caracNuc}. Note that $\lambda^1(A)$ is the space $\text{ces}(1+)$ studied in \cite{ABR18}, see our discussion at the beginning of this section.

(b) Let the K\"othe matrix $A$ be given by
\[
a_{m,n}=(\log(n+1))^{m},\quad m\geq 1, n\geq 0. 
\]
Then the corresponding K\"othe sequence spaces are power series spaces of infinite type with $\alpha_n=\log(\log(n+1))$, $n\geq 2$. On any such space there are frequently hypercyclic weighted shifts that are not chaotic since $\lim_n \frac{\alpha_{n+1}}{\alpha_n}= 1$. Since condition (N) is not satisfied, see Proposition \ref{prop-proppssinf}, all the spaces $c_0(A)$ and $\lambda^p(A)$, $1\leq p <\infty$, are different by Proposition \ref{caracNuc}.
\end{example}

\end{document}